\documentclass[10pt,reqno]{amsart}
\title{Non-uniform hyperbolicity and existence of 
absolutely continuous invariant measures}
\author{Javier Solano} 
\address{IMPA, Estrada Dona Castorina 110, Rio de Janeiro, 22460-320,
Brazil}
\curraddr{Universidade Federal Fluminense, Instituto de Matem\'atica, 
Rua M\'ario Santos Braga, s/n, Valonguinho, Niter\'oi, 24020-140
Brazil}
\email{jsolano@impa.br}
\keywords{non-uniform hiperbolicity, absolutely continuous invariant
measures}
\thanks{The author was supported by CNPq and CAPES, Brazil}
\usepackage{inputenc}
\usepackage{enumerate}
\usepackage{amsmath,amsthm,amsfonts,amssymb,amstext,amsfonts,amscd,
latexsym}
\usepackage{mathrsfs} 
\usepackage{plain}
\usepackage{pxfonts} 
\usepackage[english]{babel}
\usepackage{fontenc}
\usepackage{textcomp}
\usepackage{color}
\DeclareMathAlphabet{\mathpzc}{OT1}{pzc}{m}{it}
\def \al {\alpha }  \def \de {\delta }
\def \te {\theta} \def \fhi {\varphi} \def \ep {\epsilon} 
\def \ga {\gamma} \def \la {\lambda} \def \om {\omega}
  
\def \constantB {\mbox{\large$\varsigmaup$}} 
\def \secondelta {\sigma} 
\def \thirdelta {\varepsilon} 
\def \constantdomi {a} 
\def \bigchi {\mbox{\Large$\chi$}} 
\def \crit {\mathscr{C}} 
\def \cc {\mathcal{S}} 
\def \ra {\rightarrow }

 \def \slim {\limsup_{n\ra\infty}\, }
\def \ilim {\liminf_{n\ra\infty}\, }
 
\def \soma {\sum_{i=1}^{n}} 

\def \s {\textquoteright}

\newcommand{\leb}{\operatorname{Leb}}

\newcommand{\dist}{\operatorname{dist}}
\newcommand{\distv}{\operatorname{dist_{vert}}}

\newcommand{\graph}{\operatorname{graph}}

\newcommand{\toro}{\ensuremath{\mathbb{T}^1}}
\newcommand{\N}{\ensuremath{\mathbb{N}}}

\newcommand{\p}{^\prime}
\newcommand{\pcuad}{^{\prime\prime}}

\newcommand{\upla}[1]{(a_1,a_2,\ldots, a_#1)}
\newcommand{\sumaupla}[1]{a_1+a_2+\ldots+a_#1}

\theoremstyle{plain}
\newtheorem{MainThe}{Theorem}

\newtheorem{The}{Theorem}[section]
\newtheorem{Pro}{Proposition}[section]
\newtheorem{Le}{Lemma}[section]
\newtheorem{Cor}{Corollary}[section]
\newtheorem{Cla}{Claim}[section]
\newtheorem*{Conjecture}{Conjecture}
\theoremstyle{definition} 
\newtheorem{Def}{Definition}[section]
\newtheorem{Rem}{Remark}[section]

\newtheorem*{Proof}{Proof}
\newtheorem*{Acknowledge}{Acknowledgements}
\theoremstyle{remark}
\numberwithin{equation}{section}

\begin{document}
\maketitle

\begin{abstract}
We prove that for certain partially hyperbolic skew-products, 
non-uniform hyperbolicity along the leaves implies existence of a
finite number of ergodic absolutely continuous invariant probability
measures which describe the asymptotics of almost every point. The
main technical tool is an extension for sequences of maps of a result
of de Melo and van Strien relating hyperbolicity to recurrence
properties of orbits. As a consequence of our main result, we also
obtain a partial extension of Keller\textquoteright s theorem
guaranteeing the existence of absolutely continuous invariant measures
for non-uniformly hyperbolic one dimensional maps.
\end{abstract}  

\begin{section}{Introduction}
\markboth{}{INTRODUCTION}
In this paper we study the existence of absolutely continuous
invariant probability measures for non-uniformly expanding maps in
dimensions larger than 1.

It is a classical fact (see Ma\~n\'e, \cite{Ma}) 
that every uniformly expanding smooth map on a compact manifold admits
a unique ergodic absolutely continuous invariant measure, and this
measure describes the asymptotics of almost every point. Moreover, see
Bowen \cite{Bow}, uniformly  hyperbolic diffeomorphisms also have a
finite number of such \emph{physical measures}, describing the
asymptotics of almost every point. Actually, in this case, the
physical measures are absolutely continuous only along certain
directions, namely, the expanding ones.

The present work is motivated by the question of knowing, 
to what extent, weaker forms of hyperbolicity are still sufficient for
the existence of such measures. A precise statement in this direction
is:

\begin{Conjecture}[Viana, \cite{Via2}]
If a smooth map has only non-zero Lyapunov exponents at 
Lebesgue almost every point, then it admits some physical measure.
\end{Conjecture}

Two main results provide some evidence in favor of this conjecture. 
The older one is the remarkable theorem of Keller \cite{Keller}
stating that \emph{for maps of the interval with finitely many
critical points and non-positive Schwarzian derivative, existence of
absolutely continuous invariant probability is guaranteed by positive
Lyapunov exponents},  i.e.,
\begin{equation} \label{limsup>0}
\slim \frac{1}{n} \log |Df^n(x)|>0 
\end{equation}
\emph{on a positive Lebesgue measure set of points $x$} 
(see Subsection \ref{Preliminariesonedim} for definitions involved).
In fact, Keller proved the existence of a finite number of these
measures whose union of basins have full Lebesgue measure, in the case
that (\ref{limsup>0}) holds for Lebesgue almost every point.

Then, more recently, Alves, Bonatti and Viana \cite{ABV} proved that 
\emph{every non-uniformly expanding local diffeomorphism on any
compact manifold admits a finite number of ergodic absolutely
continuous invariant measures describing the asymptotics of almost
every point}. This notion of non-uniform expansion means that
\begin{equation} \label{liminf>0}
\ilim \frac{1}{n} \sum_{j=0}^{n-1} \log ||Df(f^j(x))^{-1}||^{-1} \geq
c>0 
\end{equation}
almost everywhere. Alves, Bonatti and Viana \cite{ABV} also give a
version of this result for maps with singularities, that is, which
fail to be a local diffeomorphism on some subset $\cc$ of the ambient
manifold. However, due to the presence of singularities they need an
additional hypothesis (of slow recurrence to the singular set $\cc$)
which is often difficult to verify. 
Given that Keller\textquoteright s theorem has no hypothesis about the
recurrence to the singular set (in his case $\cc=\{\text{critical
points}\}$), one may ask to what extent this condition is really
necessary. 

\markboth{}{}
This question was the starting point of the present work. 
Before giving our statements, let us mention a few related results.

One partial extension of both Keller \cite{Keller} and Alves, 
Bonatti and Viana \cite{ABV}, was obtained recently by Pinheiro
\cite{Pin}: he keeps the slow recurrence condition but is able to
weaken the hyperbolicity condition substantially, replacing $\liminf$
by $\limsup$ in (\ref{liminf>0}).

Another important result was due to Tsujii \cite{Ts}: $C^r$ generic 
partially hyperbolic endomorphisms on a compact surface admit finitely
many ergodic physical measures and the union of their basins is a
total Lebesgue measure set. When the center Lyapunov exponents are
positive, these measures are absolutely continuous.

Our own results holds for a whole, explicitly defined, 
family of transformations on surfaces. We prove existence and
finiteness of ergodic absolutely continuous invariant measures,
assuming only non-uniform expansion (slow recurrence is not
necessary).

Motivated by a family of maps introduced by Viana \cite{Via1} 
and studied by several other authors (see for example
\cite{Alv,AV,BST,Sc,AA}) we consider transformations of the form
$\fhi:\toro\times I_0\to \toro\times I_0$, $(\te,x)\mapsto (g(\te),
f(\te,x))$, where $g$ is a uniformly expanding circle map, each
$f(\te,\cdot)$ is a smooth interval map with non-positive Schwarzian
derivative, and $\fhi$ is partially hyperbolic with vertical central
direction:
\begin{equation*}
|\partial_\te g(\te)|>|\partial_x f(\te,x)| \quad\quad \text{ at all
points}.
\end{equation*}
We prove that if $\fhi$ is non-uniformly expanding then it admits 
some absolutely continuous invariant probability. Moreover, there
exist finitely many ergodic absolutely continuous invariant
probabilities whose union of basins is a full Lebesgue measure set.

The Viana maps \cite{Via1} correspond to the case when $g$ is affine, 
$g(\te)=d\te$ (mod 1) with $d>>1$, and $f$ has the form
$f(\te,x)=a_0+\al \sin (2\pi \te)- x^2$ (actually, \cite{Via1} deals
also with arbitrary small perturbations of such maps).
It was shown in \cite{Via1} that Viana maps are indeed non-uniformly 
expanding. Moreover, Alves \cite{Alv} proved that they have a unique
physical measure, which is absolutely continuous and ergodic. Their
methods hold even for a whole open set of maps not necessarily of
skew-product form. In fact, the argument of \cite{Alv} rely on a proof
of slow recurrence to the critical set which in that case is the
circle $\toro\times \{0\}$.

For the family of maps which we consider (see Theorem
\ref{PrincipalA}), we do not assume the slow recurrence condition,
fundamental in \cite{ABV}, \cite{Alv} and \cite{Pin}. On the other
hand, our method is completely different from the one used in the
mentioned works. We view $\fhi$ as a family of smooth maps of the
interval, namely, its restrictions to the vertical fibers
$\{\te\}\times I_0$. Thus, our main technical tool is an extension for
such families of maps of a result proved by de Melo and van Strien
\cite[Theorem V.3.2, page 371]{dMvS} for individual unimodal maps
saying, in a few words, that positive Lyapunov exponents manifest
themselves at a macroscopic level: intervals that are mapped
diffeomorphically onto large domains under iterates of the map. This,
in turn, allows us to make use of the hyperbolic times technique
similar to the one introduced by Alves, Bonatti and Viana \cite{ABV}.

Let us remark that in the setting of piecewise expanding maps 
in high dimensions, there are several works which deal with existence
of absolutely continuous invariant measures. Among them, let us
mention \cite{Ad, B, GB, Keller2, S}. In all the cases, additional
conditions on the expanding constants and (or) the boundary behavior
are required.
\begin{subsection}{Organization of the paper}

This paper is organized as follows. In Section \ref{Statements} we
give the precise statement of the main results. In section
\ref{Preliminaries} we introduce a few preliminary facts, which will
be useful in the sequel. In section \ref{Composition} we prove our
Theorem \ref{PrincipalB}, which is the extension of \cite[Theorem
V.3.2, page 371]{dMvS} mentioned before. The section
\ref{Consequences} contains the proof of one partial extension of
Keller\textquoteright s theorem. 

In section \ref{Hyperbolic-like} we prove another key result 
(Proposition \ref{difeoboundist}): for each interval which is mapped
diffeomorphically onto a large domain under an iterate of the
skew-product, there exists an open set containing this interval which
is sent diffeomorphically onto its image under the same iterate.
Moreover, this map has bounded distortion and the measure of the image
is bounded away from zero. We call these iterates
\emph{hyperbolic-like times}, because their behavior is similar to
hyperbolic times introduced in \cite{ABV}.

In section \ref{Absolutely} we combine the main lemma 
(Lemma \ref{A_nY_n}) used in the proof of Theorem \ref{PrincipalB},
with the Pliss Lemma to conclude  that the set of points with
infinitely many (and even positive density of) hyperbolic-like times
has positive Lebesgue measure. The construction of the absolutely
conti\-nuous invariant measure for the skew-product $\fhi$ follows
along well-known lines, as we explain in subsection
\ref{constructionofmeasure}. Finally,  on subsection
\ref{ergodicityofmeasure}, we prove the ergodicity of the measure and
the existence of finitely many SRB measures.
\end{subsection}

\begin{Acknowledge}
The results of this work are essential part of my doctoral thesis 
made at IMPA. I am thankful to Marcelo Viana for advice, constant
encouragement and valuable conversations. I am indebted to Vilton
Pinheiro and Vitor Ara\'ujo for suggestions and insightful
discussions. I also thank Sebastian van Strien for readily
clarifications of important points in his book with de Melo.  
\end{Acknowledge}

\end{section}

\begin{section}{Statement of the results}\label{Statements}
Let us present the precise statements of our results.

\begin{subsection}{Non-uniformly expanding skew-products}

Let $I_0$ be an interval and let $\toro$ be the circle. 
We consider $C^3$ partially hyperbolic skew-products defined on
$\toro\times I_0$, with critical points in the vertical
direction. The mappings we consider are precisely
\begin{equation*}\label{skewproduct}
\begin{array}{ccccc}
\fhi&:&\toro\times I_0&\ra&\toro\times I_0\\
&&(\te, x)&\ra&(g(\te), f(\te,x))
\end{array}
\end{equation*}
where $g$ is a uniformly expanding smooth map on $\toro$ and 
$f_{\te}:I_0\ra I_0\,,\:x\ra f(\te, x)$
is a smooth map, possibly with critical points, for every $\te\in
\toro$. We assume our map is partially hyperbolic, it means that 
satisfies (\ref{decomdomin1}) below (see Subsection
\ref{Preliminariestwodim}).

In the result of Alves, Bonatti and Viana (see \cite[Theorem C]{ABV}),
the set $\cc$ of singular points of $\fhi$ satisfies the
non-degenerate singular set conditions. These conditions allow the
co-existence of critical points and points with $|\det D\fhi| 
=\infty$. We will only admit critical points. 

We denote by $\crit$ the set of critical points of $\fhi$ 
and by $\crit_\te$ the  set of critical points contained in the
$\te$-vertical leaf. By $\distv$ we denote the distance induced by the
Riemmanian metric in the vertical leaf, i.e, if $z=(\te,x)$ for some
$x$, $\distv(z,\crit)=\dist(z,\crit_\te)$. 

Let $M=\toro\times I_0$ and $\crit\subset M$ a compact set. 
We consider a $C^3$ skew product map $\fhi:M\ra M$ which is a local
$C^3$ diffeomorphism in the whole manifold except in a critical set
$\crit$ such that: 
\begin{enumerate}
\item[($F_1$)]
$p=\sup \#\crit_\te<\infty\:$;
\end{enumerate}
there exists $B>0$ such that, for every $z\in M\setminus \crit$, 
$w\in M$ with $\dist(z,w)<\distv(z,\crit)/2$
\begin{enumerate}
\item[($F_2$)]
\quad $\displaystyle{\left|\log |\partial_x f(z)|-
\log|\partial_x f(w)|
\:\right|\leq \frac{B}{\distv(z,\crit)}\dist(z,w)}\:$;
\end{enumerate}
and for all $\te\in\toro$
\begin{enumerate}
\item[$(F_3)$] $Sf(\te,x)\leq 0$, for $x\in I_0$ where this 
quantity is defined.
\end{enumerate}
When $M=I_0$, if $f$ satisfies the one dimensional definition of 
non-flatness and $Sf\leq 0$ (see subsection \ref{Preliminariesonedim}
for definitions), then it automatically satisfies these conditions.
Now, we are in position to state our main result.
\begin{MainThe} \label{PrincipalA}
Assume that $\fhi:\toro\times I_0\rightarrow \toro\times I_0$ is 
a $C^3$ partially hyperbolic skew product satisfying $(F_1)$, $(F_2)$
and $(F_3)$. If $\fhi$ is non-uniformly expanding, i.e, for Lebesgue
almost every $z\in \toro\times I_0$,
\begin{equation}\label{nue}
{\liminf}_{n\rightarrow\infty}\frac{1}{n}\sum_{j=0}^{n-1}
\log \|D\fhi(\fhi^j(z))^{-1}\|^{-1}>0\: ,
\end{equation}
then $\fhi$ admits an  absolutely continuous invariant measure. 
Moreover, if the limit in (\ref{nue}) is bounded away from zero, then
there exist finitely many ergodic absolutely continuous invariant
measures and their basins cover $M$ up to a zero Lebesgue measure set.
\end{MainThe}

\begin{Rem} 
\noindent
\begin{enumerate}[(i)]
\item Note that $(F_2)$ implies that for any $z\in M$, 
$\dist(z,\crit)\geq \frac{\distv(z,\crit)}{2}$. 
\item If $\crit_\te=\emptyset$ for some $\te\in\toro$ then, as 
a consequence of $(F_2)$, $\crit_\te=\emptyset$ for every
$\te\in\toro$. This case is covered by \cite[Corollary D]{ABV}, but
also follows from (a simple version of) our arguments. For
completeness we define $\dist(z,\emptyset)=1$.
\item When the critical set $\crit$ is such that $\dist(z,\crit)\geq
\eta\distv(z,\crit)$ for all $z\in M$ and some $\eta>0$, then we may
replace $\distv$ by $\dist$ in the condition $(F_2)$.
\end{enumerate}
\end{Rem}
\end{subsection}

\begin{subsection}{Sequences of smooth one dimensional maps}
\label{statement2}
In order to prove Theorem \ref{PrincipalA}, we analyze the 
dynamics of the transformation along the family of vertical leaves.
The main technical point is to bound the distortion of the iterates
along suitable subintervals of the leaves. The precise statement is
given in Theorem \ref{PrincipalB}. Beforehand, we need to introduce
some notations.

Given an interval $I_0$, let us consider a sequence $\{f_k\}_{k\geq
0}$ of $C^1$ maps $f_k:I_0\ra I_0$. Let us denote by $\crit_k$ the set
of critical points of $f_k$, for every $k\geq 0$.  Notice that
$\crit_k$ could be an empty set for any $k\in\N$. We are interested on
the study of the dynamics given by the compositions of maps in the
sequence. Thus, we define for $i\geq 1$ and $x\in I_0$,
\begin{equation*}
f^i(x)=f_{i-1}\circ\ldots \circ f_{1}\circ f_{0}(x)
\end{equation*} 
and we denote $f^0(x)=x$ for $x\in I_0$.

Based on the definitions of $T_i(x)$ and $r_i(x)$ on the case 
that there are just iterates of a function (see for instance
\cite[page 335]{dMvS}), we define for $i\in\N$ and $x\in I_0$:
\begin{align*}
T_i\left(\{f_k\},x\right)&:=\text{Maximal interval contained in}
\:  I_0, \text{containing}\: x, \\
&\phantom{:=:} \text{ such that}\: f^j(T_i(x))\cap\crit_{j}=\emptyset
\:\: \text{for}\: 0\leq j< i \:;\\
L_i\left(\{f_k\},x\right), R_i\left(\{f_k\},x\right)&:=
\text{Connected components of}\:\: T_i\left(\{f_k\},x\right)\setminus
\{x\}\:;\\
r_i\left(\{f_k\},x\right)&:=\min 
\left\{\,\left|f^i\left(L_i\left(\{f_k\},x\right)\right)\right|,
\left|f^i\left(R_i\left(\{f_k\},x\right)\right)\right|\,\right\}.
\end{align*}
When it does not lead to confusion, we denote these functions just by
$T_i(x)$, $L_i(x)$, $R_i(x)$, $r_i(x)$. In this subsection and in the
proof of the results of this subsection, we will use this simplified
notation, since the sequence $\{f_k\}$ is fixed.

Our goal is to show that positive Lyapunov exponents imply that 
the average of the $r_i$ is positive. 
We consider a sequence $\{f_k\}$ with positive Lyapunov exponents. 
Namely,  $\{f_k\}$ satisfies the following condition: there exists
$\la>0$ such that 
\begin{equation} \label{liminf>2la}
\displaystyle{\ilim \frac{1}{n} \log |Df^n(x)|>2\lambda}
\end{equation}
for every $x$ in some subset of $I_0$.

The following compactness condition on the sequence of maps
$\{f_k\}_{k\geq 0}$, together with positive Lyapunov exponents,
guarantee the positiveness of the average of the $r_i$.

Recall that a sequence $\{f_k\}_k$ of $C^1$ maps $f_k:I_0\ra I_0$ 
is said to be \emph{$C^1$-uniformly equicontinuous} if, given
$\zeta>0$, there exists $\ep>0$ such that
\begin{equation} \label{uniformlyequicontinuous}
|x-y|<\ep \quad \text{ implies } \quad  \left\{\begin{gathered}
|f_k(x)-f_k(y)|<\zeta \\
|Df_k(x)-Df_k(y)|<\zeta
\end{gathered}
\right.
\end{equation} 
for all $k\in\N$. Recall also that a sequence $\{f_k\}_k$ of $C^1$ 
maps $f_k:I_0\ra I_0$ is said to be \emph{$C^1$-uniformly bounded} if
there exists $\Gamma>0$ such that for every $x\in I_0$,
\begin{equation} \label{uniformlybounded} 
|f_k(x)|\:, \:|Df_k(x)|\leq \Gamma 
\end{equation}
for all $k\in\N$.

Our main result in this setting is the following.
\begin{MainThe} \label{PrincipalB}
Let $\{f_k\}$ be a $C^1$-uniformly equicontinuous and $C^1$-uniformly 
bounded sequence of smooth maps $f_k:I_0\ra I_0$ for which $p=\sup_k
\#\crit_k<\infty$, and (\ref{liminf>2la}) holds for all $x$ in a set
$H$, for some $\la>0$. Then, there exists $\constantB>0$ such that
\begin{equation} \label{liminfr_i}
\ilim \frac{1}{n} \soma r_i\left(\{f_k\},x\right)\geq\constantB 
\end{equation}
for Lebesgue almost every $x\in H$.
\end{MainThe}
\begin{Rem}
We do not require that $f_k$ be a multimodal map, for any $k\geq 0$. 
The non-positive Schwarzian derivative condition is not necessary. 
\end{Rem}

This result may be viewed as a ``random'' version of 
Theorem V.3.2 (page 371) in de Melo, van Strien \cite{dMvS}. Notice
however, that this does not follow from the result of de Melo and van
Strien because the dynamics of the maps we consider is more
complicated. For example, in the unimodal case the hypothesis ensures
that the critical point is not periodic, in our context one can not
prevent the iterates of the critical set from intersecting the
critical set.

Notice that in the setting of Theorem \ref{PrincipalA}, 
the result of Theorem \ref{PrincipalB} is applied to the restrictions
of $\fhi$ to the orbits of the vertical leaves.

The result of Theorem \ref{PrincipalB} still holds replacing
$\liminf$ by $\limsup$.
\begin{Cor} \label{MainCor} 
Let $\{f_k\}$ be a $C^1$-uniformly equicontinuous and $C^1$-uniformly 
bounded sequence of smooth maps $f_k:I_0\ra I_0$ for which $p=\sup_k
\#\crit_k<\infty$, and there exists $\la>0$ such that 
\begin{equation} \label{limsup>2la}
\displaystyle{\slim \frac{1}{n} \log |Df^n(x)|>2\lambda}
\end{equation}
for all $x$ in a set $H$. Then, there exists $\constantB>0$ such that
$\slim \frac{1}{n} \soma r_i\left(\{f_k\},x\right)\geq\constantB \:$, 
for Lebesgue almost every $x\in H$.
\end{Cor}
In the case that the sequence $\{f_k\}_{k\geq 0}$ is constant 
($f_k=f$, for all $k\geq 0$), we obtain the following result for
multimodal maps. For definitions involved, see
Subsection~\ref{Preliminariesonedim}.

\begin{Cor} \label{KellerwithoutSf}
Let $f:I_0\ra I_0$ be a $C^3$ multimodal map with non-flat critical
points. Assume that $f$ does not have neutral periodic points. 
If (\ref{limsup>0}) holds for Lebesgue almost every point, 
then there exists an absolutely continuous invariant measure. 
\end{Cor}
Notice that the hypothesis is weaker than in Keller \cite{Keller}, 
because we make no assumption on the Schwarzian derivative. On the
other hand, we only prove existence (not finiteness) of the absolutely
continuous invariant measure. 

In particular, for $C^3$ multimodal maps with non-flat critical 
points and with eventual negative Schwarzian derivative (i.e, there
exists $k\in \N$ such that $f^k$ has negative Schwarzian
derivative), positive Lyapunov exponents implies the existence 
of an absolutely continuous invariant measure. Indeed, since  for
these class of maps, the neutral periodic points are attracting points
(see \cite[Theorem 2.5]{Webb}), we can apply Corollary
\ref{KellerwithoutSf}. 
\end{subsection}
\end{section}

\begin{section}{Preliminary results} \label{Preliminaries} 
We first recall some well-known properties and 
tools for one dimensional maps to be used in this work. 

\begin{subsection}{One dimensional dynamics}
\label{Preliminariesonedim}
Let $I$ be an interval and let $f:I\to I$ be a differentiable map. 
A point $c\in I$ is called a \emph{critical point} if $f\p(c)=0$. A
map is called \emph{smooth} if it is at least a $C^1$ map with any
number (possibly zero) of critical points. A map is called
\emph{multimodal} if it is a smooth map and there is a partition of
$I$ in finitely many subintervals on which $f$ is strictly monotone.
It is called \emph{unimodal} if the partition has exactly two
subintervals. Without loss of generality it is assumed that for a
multimodal map $f$, $f(\partial I)\subset \partial I$. Let
$c_1,\ldots, c_d$ be the critical points of $f$. We say that the
critical point $c_i$ is $C^n$ \emph{non-flat of order $l_i>1$} if
there exist a local $C^n$ diffeomorphism $\phi_i$ with
$\phi_i(c_i)=0$, such that near $c_i$, $f$ can be written as
\begin{equation*}
f(x)=\pm |\phi_i(x)|^{l_i}+ f(c_i).
\end{equation*}
The critical point is \emph{$C^n$ non-flat} if it is $C^n$ non-flat 
of order $l_i$ for some $l_i>1$. In all that follows, we will just say
that $c_i$ is a non-flat critical point of a $C^n$ multimodal map $f$
if $c_i$ is a $C^n$ non-flat critical point. Here $n=3$ is enough for
Corollary \ref{KellerwithoutSf}.

When the map $f$ is $C^3$ (or three times differentiable) we can
define
\begin{equation*}
Sf(x)=\frac{f^{\prime\prime\prime}(x)}{f^{\prime}(x)}-
\frac{3}{2}\left(\frac{f^{\prime\prime}(x)}{f^{\prime}(x)}\right)^2
\end{equation*}
for $x$ such that $f^{\prime}(x)\neq 0$. This quantity is called 
the \emph{Schwarzian derivative} of $f$ at the point $x$. There are
many results for one dimensional dynamics that are only known for
those maps whose Schwarzian derivative is non-positive. 

One standard way to prove the existence of absolutely continuous
invariant measures for $f$ is to define a Markov map associated to $f$
and take advantage of the known fact of the existence of this kind of
measures for Markov maps.
\begin{Def} \label{Markovmap}
We call a map $F:J\ra J$ \emph{Markov} if there exists a countable 
family of disjoint open intervals $\{J_i\}_{i\in\N}$ with
$\leb(J\setminus\cup J_i)=0$, such that:
\begin{enumerate}
\item[$(M_1)$] there exists $K>0$ such that for every $n\in\N$ 
and every $T$ such that $F^j(T)$ is contained in some $J_i$ for
$j=0,1,\ldots, n$, it holds
\begin{equation*}
\frac{|DF^n(x)|}{|DF^n(y)|}\leq K \quad\text{ for } x,y\in T ;
\end{equation*}
\item[$(M_2)$] if $F(J_k)\cap J_i\neq\emptyset$ then $J_i\subset
F(J_k)$;
\item[$(M_3)$] there exists $r>0$ such that $|F(J_i)|\geq r$ for all
$i$.
\end{enumerate}
\end{Def} 

Condition $(M_1)$ is known as bounded distortion. Given open intervals
$J\subset T$, let $L$, $R$ be the connected components of $T\setminus
J$. We say that $T$ is a \emph{$\kappa$-scaled} neighborhood of $J$ if
both connected components of $T\setminus J$ have length $\kappa |J|$.
We define $b(T,J)=|J||T|/|L||R|$, and when $f$ is monotone continuous,
$B(f,T,J)=b(f(T),f(J))/b(T,J)$ (this is known as \emph{cross ratio
operator}). Koebe Principle claims that the control of cross ratio
operator plus $\kappa$-scalation (for some $\kappa>0$) imply bounded
distortion (see \cite[Theorem IV.1.2]{dMvS}). When $Sf\leq 0$, cross
ratio satisfies the condition required on Koebe Principle.
In order to control the distortion when we consider iterates of a
single map without Schwarzian derivative assumptions, we use the next
result. Recall that a periodic point $p$ of period $k$ is
\emph{repelling}  if $|Df^k(p)|> 1$, \emph{attracting} if $|Df^k(p)|<
1$ and \emph{neutral} if $|Df^k(p)|=1$.
The proof of the result follows from \cite[Theorem B]{Koz} for the
unimodal case, and \cite[Theorem C]{SV} for the multimodal case. The
hypothesis of these theorems are less restrictive than ours.
\begin{The}[] \label{TheoremKoz}
Let $f:I\to I$ be a $C^3$ multimodal map with non-flat critical
points. Assume that the periodic points of $f$ are repelling. Then,
there exists $C>0$ such that if $I\subset M$ are intervals and
$f^n_{|M}$ is a diffeomorphism,
\begin{equation*}
B(f^n,M,I)\geq \exp (-C|f^n(M)|^2).
\end{equation*}
\end{The}
Finally let us state the following theorem which we use in the proof
of Corollary \ref{KellerwithoutSf}. Recall that an interval $J\subset
I$ is called a \emph{wandering interval} for $f:I\ra I$ if the
intervals $J, f(J),\ldots $ are pairwise disjoint and the images
$f^n(J)$ do not converge to a periodic attractor when $n\ra\infty$. De
Melo, van Strien and Martens \cite{MSM} proved that,
if $f:I\ra I$ is a $C^2$ map with non-flat critical points then $f$
has no wandering interval. Recall also that the Lebesgue measure is
said to be \emph{ergodic} for $f:J\ra J$, if for each $X\subset J$
such that $f^{-1}(X)=X$, one of the sets $X$ or $\complement X$ have
full Lebesgue measure.   

\begin{The} \label{C^3maps}
Let $f:I\ra I$ be a $C^3$ map without wandering intervals and 
with all the periodic points repelling (i.e, $f$ does not have either
attracting or neutral periodic points). Then:
\begin{enumerate}[(i)]
\item the set of preimages of the critical set $\crit$ is dense in
$I$. 
\end{enumerate}
Moreover, if the map $f$ is multimodal then:
\begin{enumerate}[(i)]
\addtocounter{enumi}{1}
\item every non-wandering critical point is approximated by periodic
points;
\item if the critical points are non-flat: there are finitely many
forward invariant sets $X_1,\ldots, X_k$ such that $\cup B(X_i)$ has
full measure in $I$, and $f_{| B(X_i)}$ is ergodic with respect to the
Lebesgue measure  (here, $B(X_i)=\{y ; \omega(y)=X_i\}$ is the basin
of $X_i$). In the unimodal case we have $k=1$, so $f$ is ergodic with
respect to Lebesgue measure.
\end{enumerate}
\end{The}
The proof of item $(i)$ follows from standard arguments. For item
$(ii)$, see
\cite{Yo}. The proof of item $(iii)$ is contained in the proof of
Theorem E of \cite{SV}. 

On our Theorem \ref{PrincipalB} we adapt some tools used on 
one dimensional dynamics: given a smooth map $f:I_0\ra I_0$ and $x\in
I_0$, for every $n\in\N$, let $T_n(x)$ be, the maximal interval
containing $x$ where $f^n$ is a diffeomorphism. Let $r_n(x)$ be the
length of the smallest component  of $f^n(T_n(x))\setminus f^n(x)$.
Koebe Principle guarantees distortion bounds in the orbit of a point
$x$, if the respective $r_n(x)$ are not too small. Of course, a lower
bound on $r_n(x)$ implies that the images of the monotonicity
intervals are not too small. This gives some idea of the importance of
the result of Theorem \ref{PrincipalB}.

\end{subsection}

\begin{subsection}{Partial hyperbolicity, slow recurrence}
\label{Preliminariestwodim}
We call a $C^1$ mapping $\fhi:M\ra M$ \emph{partially hyperbolic} 
endomorphism if there are constants $0<\constantdomi<1$,  $C>0$ and a
continuous decomposition of the tangent bundle $TM=E^c\oplus E^u$ such
that: 
\begin{enumerate}
\item[(a)] $||D\fhi^n(z) \,v||>C^{-1} \constantdomi^{-n}$, 
for every unit vetor $v\in E^u(z)$.
\item[(b)] $||D\fhi^n(z) \,u||<C \constantdomi^n ||D\fhi^n(z) \,v||$, 
for every pair of unit vectors $u\in E^c(z)$ and $v\in E^u(z)$.
\end{enumerate}
for all $z\in M$ and $n\geq 0$. The subbundle $E^c$ is called 
\emph{central} and the $E^u$ is called \emph{unstable}. Observe that
we do not ask invariance of the subbundles.  For the skew-product maps
that we consider, the central subbundle is given by the vertical
direction. The unstable one is given by the horizontal direction. 
Notice that the partial hyperbolicity property in our skew-product 
context means that for all $(\te, x)\in \toro\times I_0$ and $n\in\N$,
\begin{equation} \label{decomdomin1}
 \frac{\prod_{i=0}^{n-1}  
|\partial_x f(\fhi^i(\te,x))|}{|\partial_\te g^n(\te)|}\leq C
\constantdomi^n .
\end{equation}
Let us remark that in the condition $(F_2)$ of Theorem
\ref{PrincipalA} we may put $\distv(z,\crit)^\ga$ (with $\ga>1$)
instead of $\distv(z,\crit)$, if we had a better domination for
$\fhi$, namely, if for all $(\te,x)\in \toro\times I_0$,
\begin{equation*}
\frac{\prod_{i=0}^{n-1} |\partial_x f(\fhi^i(\te,x))|^\ga}
{|\partial_\te g^n(\te, x)|}\leq C \constantdomi^n.
\end{equation*}
Finally, recall that the condition of slow recurrence to the 
critical set $\crit$ (see \cite[Equation (6)]{ABV}) means that given
$\epsilon>0$, there exists $\delta>0$ such that for Lebesgue almost
every $x\in M$
\begin{equation*}
\slim \frac{1}{n} \sum_{j=0}^{n-1} 
-\log \dist_\delta(\fhi^j(x),\crit)\leq \ep ,
\end{equation*}
where $\dist_\de(\fhi^j(x),\crit)=\dist(\fhi^j(x),\crit)$ if 
$\dist(\fhi^j(x),\crit)<\delta$ and $\dist_\de(\fhi^j(x),\crit)=1$
otherwise.
\end{subsection}
\end{section}

\begin{section}{Compositions of smooth one dimensional maps}
\label{Composition}
Here we prove Theorem \ref{PrincipalB}. In the sequel we 
introduce some definitions and state results whose proofs are left to
the end of the section. Theorem \ref{PrincipalB} follows from these
results. 

\begin{subsection}{Proof of Theorem \ref{PrincipalB}}
We begin by introducing some sets useful for the proof of the
theorem. Recalling the definitions in subsection \ref{statement2}, for
every $n\in\N$ and $\de>0$ we denote by, 
\begin{equation}\label{Antheta}
A_n\left(\{f_k\},\de\right):= \Biggl\{ x\in I_0\:;\: 
\frac{1}{n}\soma r_i(x)<\de^2, \:\:r_n(x)>0\Biggr\},
\end{equation}
and given $\la>0$, we define for $n\in \N$ ,
\begin{equation}\label{Yntheta}
Y_n\left(\{f_k\},\la\right):=\Biggl\{x; \frac{1}{n} 
\log |Df^n(x))|>\la\Biggr\}.
\end{equation}
When it does not lead to confusion, we denote these sets by
$A_n(\de)$ and $Y_n(\la)$. In fact, we will do it in all this section.

It is clear that (\ref{liminfr_i}) holds (for $\constantB=\de^2$) 
for Lebesgue almost every $x\in H$, if $|\cap_{n\geq N} (\complement
A_n(\de)\cap Y_n(\la))\cap H|$ converges to $|H|$, when $N\ra\infty$
(where $|B|$ denotes the Lebesgue measure of $B$ and $\complement B$
denotes the complement set of $B$). We claim that, in effect, this
happens. Indeed, for every $N\in\N$, it holds
\begin{equation*}
H\cap\left(\bigcap_{n\geq N} Y_n(\la)\right) \cap 
\complement\left(\bigcup_{n\geq N} A_n(\de)\cap Y_n(\la)\right)
\:\mbox{\Large $\subset$}\:H\cap\left(\bigcap_{n\geq N} \complement
A_n(\de)\cap Y_n(\la)\right). 
\end{equation*}
Since (\ref{liminf>2la}) holds for all $x\in H$,
$|H\cap\left(\cap_{n\geq N} Y_n(\la)\right) |$ converges to the
Lebesgue measure of $H$. Thus, to prove our claim we just need to
prove that $|\cup_{n\geq N} A_n(\de)\cap Y_n(\la)|$ converges to zero.
For this purpose we will state the following result which is the main
lemma for proving Theorem \ref{PrincipalB}. 

\begin{Le} \label{A_nY_n}
Let $\{f_k\}$ be a $C^1$-uniformly equicontinuous and $C^1$-uniformly 
bounded sequence of smooth maps $f_k:I_0\ra I_0$ for which $p=\sup_k
\#\crit_k<\infty$. Then, given $\la>0$, there exist $\de>0$ such that 
\begin{equation} \label{equationA_nY_n}
\displaystyle |A_n\left(\{f_k\},\de\right)
\cap Y_n\left(\{f_k\},\la\right)|\leq |I_0| \exp (-n\la/2)
\end{equation}
for $n$ big enough. Moreover, $\de$ depends only on $\la$, 
the modulus of continuity (\ref{uniformlyequicontinuous}), the uniform
bound $\Gamma$ in (\ref{uniformlybounded}) and the uniform bound $p$
for the number of critical points.
\end{Le}

\begin{proof}[Proof that Theorem \ref{PrincipalB} follows from 
Lemma \ref{A_nY_n}]
As we have remarked, Lemma \ref{A_nY_n} clearly implies that
\begin{equation*}
\bigcup_{N\in\N} \bigcap_{n\geq N } \complement A_n(\de)\cap Y_n(\la) 
\end{equation*}
has full Lebesgue measure in $H$. Hence, (\ref{liminfr_i}) holds 
for $\constantB=\de^2$, where $\de$ is the constant found on Lemma
\ref{A_nY_n}. This concludes the proof of Theorem \ref{PrincipalB}.
\end{proof}
\end{subsection}

\begin{subsection}{Connected components of the 
set $\mbox{\boldmath $A_n(\de)$ \unboldmath}$}
The proof of Lemma \ref{A_nY_n} relies on bounding the number of 
connected components of the set $A_n(\de)$ whose intersection with
$Y_n(\la)$ is non-empty.
We define a family of sets related to these components. It seems
easier to deal and to count the elements of this family than the
components of $A_n(\de)$, and it will be enough for our purposes. 

For $\de>0$, $a_i\in \{0,1\}$ for $i=1, 2, \ldots, n$,
\begin{equation*} 
C_{\de}(a_1,a_2,\ldots,a_n):=\{x\in  I_0\:;
\: r_i(x)\geq\de\:\: \text{if}\:\: a_i=1,  
\:\,0<r_i(x)<\de \:\:\text{if}\:\: a_i=0 \}
\end{equation*}

Note that every connected component of
$C_\de(a_1,\ldots,a_s,a_{s+1})$ is contained in a connected component
of $C_\de(a_1,\ldots,a_s)$. Moreover, every connected component of 
$C_\de(a_1,\ldots,a_s)$ is a union of connected components (with its
boundaries) of $C_\de(a_1,\ldots,a_s, a_{s+1})$. Also note (recall the
definition of $T_i(x)$ in subsection \ref{statement2}) that for every
connected component $I$ of $C_\de(a_1,\ldots,a_s)$, we have $I\subset
T_{s}(x)$ for all $x\in I$.

Given $x\in I_0$ and $n\in\N$, if $f^i(x)\notin \crit_{i}$ for $0\leq
i<n$, we can associate to it a sequence $\{a_i(x)\}_{i=1}^n$,
according to the last definition, in a natural way: 
\begin{equation*}
a_i(x)=
\begin{cases}
0 & \text{ if }\: 0<r_i(x)<\de \\
1 & \text{ if }\:\:\: r_i(x)\geq \de 
\end{cases}.
\end{equation*}
For this sequence the inequality 
$(a_1(x)+\ldots+a_{n}(x))\de\leq\soma r_i(x)$ is satisfied. In
particular, for every $x\in A_n(\de)$, the associated sequence
$\{a_i(x)\}_{i=1}^{n}$ is such that $a_1(x)+\ldots+a_{n}(x)< \de n$.
Therefore, if we define
\begin{equation*}
C_n(\de):=\bigcup_{a_1+\ldots+a_n<\de n} C_{\de}(a_1,\ldots,a_n),
\end{equation*}
we conclude that $A_n(\de)\subset C_n(\de)$.

But in fact, we are interested on the connected components of 
$A_n(\de)$ which intersect the set $Y_n(\la)$. We will say that a
connected component $J$ of $A_n(\de)$ is a connected component of
$A\p_n(\de)$ if $J\cap Y_n(\la)\neq\emptyset$. Analogously we will say
that a connected component $I$ of $C_{\de}\upla{n}$ is a connected
component of $C\p_{\de}\upla{n}$ if $I\cap Y_n(\la)\neq\emptyset$. 

We can associate to each connected component of $A_n\p(\de)$, 
a connected component of $C\p_{\de}\upla{n}$, where $\sumaupla{n}<\de
n$: for a connected component $J$ of $A_n\p(\de)$, there exist
$a_1,\ldots, a_n$ (such that $\sumaupla{n}<\de n$) and a connected
component $I$ of $C\p_{\de}\upla{n}$, for which $J\cap I\neq
\emptyset$. Indeed, we can consider $a_i=a_i(x)$ ($1\leq i\leq n$) for
$x\in J\cap Y_n(\la)$, and $I$ the connected component of
$C\p_{\de}\upla{n}$ which contains $x$. Thus, we associate to $J$ the
component $I$. 

We would like to bound the number of connected components of 
$A_n\p(\de)$ by the number of connected components of
$C\p_{\de}\upla{n}$, varying $a_1,\ldots, a_n$ such that
$\sumaupla{n}<\de n$. But every connected component of
$C_{\de}\p\upla{n}$ (with $\sumaupla{n}<\de n$) could intersect more
than one connected component of $A_n\p(\de)$. By this reason we define
the following set:
\begin{equation*}
A_n\pcuad(\de):= \bigcup_{J\p\in A_n\p(\de)} J\pcuad,
\end{equation*}
where 
\begin{equation*}
J\pcuad :=J\p \quad\cup \bigcup_{a_1+\ldots+a_n<\de
n}\{\text{connected components of } C_\de(a_1,\ldots,a_n) \text{ which
intersect } J\p\cap Y_n(\la)  \} 
\end{equation*}
Obviously, a connected component of $A_n\pcuad(\de)$ could contain 
more than one connected component of $A_n\p(\de)$. However, the
restriction of $f^n$ to every connected component of $A_n\pcuad(\de)$
is a diffeomorphism. Using this fact, we will show in the proof of
Lemma \ref{A_nY_n} that in order to obtain (\ref{equationA_nY_n}), it
is enough to estimate the number of connected components of
$A_n\pcuad(\de)$.

Since every component of $A_n\pcuad(\de)$ intersect at least 
one component of $C_\de\p(a_1,\ldots,a_n)$, we conclude that
\begin{equation}\label{A_n<C}
\#\: A\pcuad_n(\de)\leq \sum \#\:  C\p_{\de}(a_1,\ldots,a_n)
\end{equation}
where the sum is over all $a_1,\ldots, a_n$ such that 
$a_1+\ldots+a_n <\de n$, and $\#X$ denotes the number of connected
components of $X$.

As we have said, Lemma \ref{A_nY_n} is a consequence of the following
result, which gives an estimate of the number of connected components
of $A_n\pcuad(\de)$ .

\begin{Le} \label{lemaprin}
Given $\la>0$, there exists $\de>0$ such that the number of 
connected components of $A_n\pcuad(\de)$ is less than $\exp (n\la/2)$.
Moreover, $\de$ depends only on $\la$, the modulus of continuity
(\ref{uniformlyequicontinuous}), the uniform bound $\Gamma$ in
(\ref{uniformlybounded}) and the uniform bound $p$ for the number of
critical points.
\end{Le}
\end{subsection}

\begin{subsection}{Consequences of expansion and continuity}
For the proof of Lemma \ref{lemaprin} we will use several results 
that we state now. First we give some notations. Given $\ep>0$, for
every $k\geq 0$, we call $V_{\ep}\crit_k$ a neighborhood of $\crit_k$
defined as the union of all $B(x,\ep)$ (ball centered in $x$ of ratio
$\ep$) varying $x\in\crit_k$. In order to simplify the notation we say
that $f^j(x)\in V_{\ep}\crit$ if $f^j(x)\in V_{\ep}\crit_{j}$ for any
$j\in\N$. The next lemma asserts that for points in $Y_n(\la)$, the
frequency of visits to the neighborhood $V_{\ep}\crit$ can be made
arbitrarily small, if $\ep$ is chosen small enough. 
\begin{Le} \label{lemafrequence}
Given $\gamma>0$, there exists $\ep>0$, such that for $x\in Y_n(\la)$,
\begin{equation*} \label{frequence}
\displaystyle
\frac{1}{n}\sum_{j=0}^{n-1}\bigchi_{V_{\ep}\crit}(f^j(x))<\gamma .
\end{equation*}
Moreover, $\ep$ does not depend on $n$, but it depends on $\la$, 
on the modulus of continuity of $\{f_k\}$ and on the uniform bound of
$\{Df_k\}$.
\end{Le}
\begin{proof} 
Using the fact that the sequence $\{f_k\}_{k\geq 0}$ is
$C^1$-uniformly equicontinuous, we conclude that given $\zeta>0$,
there exists $\ep=\ep(\zeta)$ such that
\begin{equation} \label{byequicontinuity}
|x-\crit_k|<\ep \quad \text{  implies  }
\quad |Df_k(x)|<\zeta \quad \text{for all } \: k\geq 0 .
\end{equation}
On the other hand, since $\{f_k\}_{k\geq 0}$ is $C^1$-uniformly 
bounded, $|Df_k(x)|\leq \Gamma$ for all $k\geq 0$ and $x\in I_0$.
Thus, $\log |Df_j(f^j(x))| <\log \zeta$ if $f^j(x)\in V_{\ep}\crit$
and $\log |Df_j(f^j(x))| \leq\log \Gamma$ otherwise.

Since $\la n<\sum_{j=0}^{n-1}\log |Df_{j}(f^{j}(x))|$ for 
$x\in Y_n(\la)$ and $\log \zeta \ra-\infty$ when $\zeta\ra 0$, there
must exist $\ep$ as stated.
\end{proof}

\begin{Cor}
Assume that for Lebesgue almost every $x\in I_0$
\begin{equation*}
\ilim \frac{1}{n} \log |Df^n(x))|\geq \la>0 .
\end{equation*}
Then, given $\gamma>0$, there exists $\ep>0$, such that for
Lebesgue almost every $x\in M$,
\begin{equation*}
\slim\frac{1}{n}\sum_{j=0}^{n-1}\bigchi_{V_{\ep}\crit}(f^j(x))<\gamma
.
\end{equation*}
\end{Cor}

Let us denote for $i, j\in\N$, and $x\in I_0$,
\begin{equation*}
f_i^j(x)=f_{i+j-1}\circ \ldots f_{i+1}\circ f_{i}(x)
\end{equation*}
and $f_0^0(x)=x$. Notice that $f_0^j(x)=f^j(x)$ for $j\geq 0$ 
and $x\in I_0$. Again by the $C^1$-uniform equicontinuity of the
sequence $\{f_k\}$, we have the following property. 
\begin{Le} \label{uniformcontinuity}
Given $\ep>0$ and $l\in \N$, there exists $\de=\de(l)$ such that
\begin{equation} \label{equationuniformcontinuity} 
|x-y|\leq 2\de \quad \text{ implies } \quad |f_i^j(x)-f_i^j(y)|<\ep
\end{equation}
for all $i\geq 0$  and $0\leq j\leq l$. Moreover, $\de$ just depends 
(on $l,\ep$ and) on the modulus of continuity of $\{f_k\}$.
\end{Le}
\begin{Rem}
When $l\ra\infty$ then $\de(l)\ra 0$. Observe that we also 
have: given $\ep>0$ and $\de>0$, there exists $l=l(\de)\in\N$ such
that (\ref{equationuniformcontinuity}) holds for $0\leq j\leq l$.
\end{Rem}
From now on, $\# \{I\subset C_{\de}(a_1,\ldots,a_n); I\text{ satisfies
the property P} \}$ denotes the number of connected components of
$C_{\de}(a_1,\ldots,a_n)$ which satisfy the property P.

In order to count the components whose intersection with $Y_n(\la)$ 
is non-empty, let us decompose this set in a convenient way. Given
$\ep>0$, $m\leq n$, $\{t_1,\ldots,t_m\}\subset \{0,1,\ldots,n-1\}$, we
define
\begin{equation*}
Y_{n,\ep}(t_1,\ldots,t_m)=\{x\in Y_n(\la);
\:f^j(x)\in V_{\ep}\crit\text{ if and only if }
j\in\{t_1,\ldots,t_m\}\}
\end{equation*} 

By Lemma \ref{lemafrequence} we conclude that given $\ga>0$, 
there exists $\ep>0$ such that
\begin{equation} \label{decompositionYn}
Y_n(\la)=\cup_{m=0}^{\ga n}\cup_{t_1,\ldots, t_m} 
Y_{n,\ep}(t_1,\ldots,t_m)
\end{equation} 
where the second union is over all subsets $\{t_1, \ldots, t_m\}$ 
of $\{0,1,\ldots, n-1\}$. This together with (\ref{A_n<C}) yields,
\begin{equation} \label{A_n<Cfinal}
\# A_n\pcuad(\de)\leq \sum_{a_1,\ldots,a_n} \sum_{t_1,\ldots, t_m} 
\# \{I\subset C_{\de}(a_1,\ldots,a_n); I\cap
Y_{n,\ep}(t_1,\ldots,t_m)\neq\emptyset\}
\end{equation}
where the first sum is over all $a_1,\ldots, a_n$ such that 
$a_1+\ldots+a_n <\de n$ and the second one is over all subsets $\{t_1,
\ldots, t_m\}\subset\{0,1,\ldots, n-1\}$ with $m<\ga n$.
\end{subsection}

\begin{subsection}{Connected components of 
$\mbox{\boldmath $C_\de(a_1,\ldots, a_s)$ \unboldmath}$}
To prove Lemma \ref{lemaprin} we just need to bound the double sum 
in (\ref{A_n<Cfinal}). For this we will show some claims related to
the number of connected components of the sets $C_\de(a_1,\ldots,
a_n)$. Recall that $p$ is the maximum number of elements in any
$\crit_k$ (for $k\geq 0$). Given $I\subset I_0$ and $s\in\N$, we say
$f^s(I)\cap\crit=\emptyset$ (resp. $\neq\emptyset$) if $f^s(I)\cap
\crit_{s}=\emptyset$ (resp. $\neq\emptyset$).
\begin{Cla}\label{3(p+1)timesthelast}
For any $a_1,a_2,\ldots, a_s$ with $a_j\in\{0,1\}$ for all $j$,
\begin{equation*}
\#C_{\de}(a_1,\ldots,a_s,0)+\#C_{\de}(a_1,\ldots,a_s,1)
\leq 3(p+1)\#C_{\de}(a_1,\ldots,a_s)
\end{equation*}
\end{Cla}

\begin{Cla}\label{withoutsingularity}
Let $s, n\in\N$ and $J$ be a component of
$C_{\de}(a_1,\ldots,a_s,0)$. 
If $f^{s+i}(J)\cap \crit=\emptyset$ for $1\leq i\leq n$, then
\begin{equation*}
\#\{I\subseteq C_{\de}(a_1,\ldots,a_s, 0^{i+1}), I\subseteq J\}\leq
i+1.
\end{equation*}
for $1\leq i\leq n$, where $0^{i+1}$ means that the last $i+1$ terms 
are equal to 0.
\end{Cla}

To bound the number of connected components whose intersection 
with \linebreak $Y_{n,\ep}(t_1,\ldots,t_m)$ is non-empty, we have the
following claim.
\begin{Cla}\label{withsingularity}
Let $l\in \N$ and $\ep>0$ be constants and let $\de=\de(l)$ be 
the number given by Lemma \ref{uniformcontinuity}. For any
$a_1,\ldots, a_s$ with $a_j\in\{0,1\}$, $\{t_1,\ldots, t_m\}\subset
\{0,1,\ldots,n-1\}$. If $\{s+1,\ldots,s+i\}\cap \{t_1,\ldots,
t_m\}=\emptyset$ and $i\leq l$, then
\begin{multline*}
\#\{I\subseteq  C_{\de}(a_1,\ldots,a_s,0^{i+1}), 
I\cap Y_{n,\ep}(t_1,\ldots,t_m)\neq\emptyset \} \leq (i+1)
\#\{I\subseteq C_{\de}(a_1,\ldots,a_s,0), I\cap
Y_{n,\ep}(t_1,\ldots,t_m)\neq\emptyset \}.
\end{multline*}
\end{Cla}

\begin{proof}[Proof of Lemma \ref{lemaprin}]
We prove the lemma assuming the claims above. We have basically 
four cons\-tants, namely, $\de,\ga,\ep, l$. It is very important the
order in what they are chosen. First, we choose $l\in\N$ according to
the equation (\ref{choosel}), then we choose $\ga>0$ according to
(\ref{choosega}). Next, we find $\ep>0$, using Lemma
\ref{lemafrequence}, in such a way that (\ref{decompositionYn}) holds.
Finally, given $\ep$ and $l$, let $\de>0$ be the constant given by
Lemma \ref{uniformcontinuity} and satisfying (\ref{choosede}). 

Given $m<n$, $\de>0$ and $\ep>0$, let us consider $a_1,\ldots,a_n$ 
with $a_i\in\{0,1\}$ (such that $\sumaupla{n}<\de n$) and 
$\{t_1,\ldots, t_m\}\subset \{0,\ldots,n-1\}$. 
We can decompose the sequence $a_1\ldots a_n$ in maximal blocks of
0\textquoteright s and 1\textquoteright s. We write the symbol $\xi$
in the $j$-th position if $a_j=1$ or, $a_j=0$ and $j=t_k$ for some 
$k\in\{1,\ldots,m\}$. In this way we have,  
\begin{equation} \label{blockdecompo}
a_1 a_2\ldots a_n=\xi^{i_1}0^{j_1}\xi^{i_2}0^{j_2}\ldots
\xi^{i_h}0^{j_h}
\end{equation}
with $0\leq i_k,j_k\leq n$ for $k=1,\ldots,h$, 
$\sum_{k=1}^{h}(i_k+j_k)=n$ and $\sum_{k=1}^{h} i_k < m+\de n$. 

Lets us assume that $a_1,\ldots, a_n$ are as in (\ref{blockdecompo}). 
Let $l,\ep$ and $\de$ be as in Lemma \ref{uniformcontinuity}.  Using
claims \ref{3(p+1)timesthelast}  and \ref{withsingularity} we have,
\begin{align*}
\#\{I \subset &C_{\de}(a_0,\ldots,a_n), 
I\cap Y_{n,\ep}(t_1,\ldots,t_m)\neq\emptyset\}\leq \\
&\leq (3(p+1)(l+1)^{\frac{j_h}{l}+1}(3(p+1))^{i_h}) 
\ldots (3(p+1)(l+1)^{\frac{j_1}{l}+1}(3(p+1))^{i_1}) \\
&\leq (3(p+1))^{\sum_{k=1}^{h}i_k} 
(3(p+1))^h (l+1)^{\frac{\sum_{k=1}^{h}j_k}{l}+h} \\
&\leq (3(p+1))^{m+\de n+h} (l+1)^{\frac{n}{l}+h}.
\end{align*}

Let us remark some useful properties about the 
decomposition (\ref{blockdecompo}):
\begin{itemize}
\item if $m<\ga n$ then, since $\sumaupla{n}< \de n$, we have 
that  $\sum_{k=1}^{h}i_k < \ga n+\de n$;
\item if $\sumaupla{n}< \de n$ and $m<\ga n$, the number 
of blocks $\zeta^{i_t}0^{j_t}$ is bounded by the sum of these
quantities, i.e, $h< (\de+\ga)n+1$.
\end{itemize}
Therefore, if $\sumaupla{n}< \de n$ and $m<\ga n$ we conclude 
from the inequality above that for $n$ big enough,
\begin{equation}
\begin{split} \label{fundamental}
\#\{I&\subset C_{\de}(a_1,\ldots,a_n), 
I\cap Y_{n,\ep}(t_1,\ldots,t_m)\neq\emptyset\} \\
& \leq (3(p+1))^{\ga n+\de n}(3(p+1))^{2(\de+\ga)n} 
(l+1)^{\frac{n}{l}+2(\de+\ga)n} \leq \exp (n \:\psi_0(l,\ga,\de) )
\end{split}
\end{equation}
where $\psi_0(l,\ga,\de)=3(\de+\ga)\log (3(p+1))+ 
2(\de+\ga+\frac{1}{l})\log (2l)$. 

On the other hand, by the Stirling\textquoteright s formula, 
the number of subsets of $\{0,1,\ldots,n-1\}$ of size less than $\ga
n$ is bounded by $\exp (n(\psi_1(\ga)))$ and $\psi_1(\ga)\ra 0$ when
$\ga\ra 0$.
Therefore, from this fact and (\ref{fundamental}), we conclude
\begin{equation} \label{sumofY_n}
\sum_{t_1,\ldots, t_m} \# \{I\subset C_{\de}(a_1,\ldots,a_n); 
I\cap Y_{n,\ep}(t_1,\ldots,t_m)\neq\emptyset\}
\leq \exp (n\:\psi_2(l,\ga,\de))
\end{equation}
where the sum is over all subset 
$\{t_1, \ldots, t_m\}\subset\{0,1,\ldots, n-1\}$ with $m<\ga n$, and
$\psi_2(l,\ga,\de)=\psi_0(l,\ga,\de)+\psi_1(\ga)$.

Once again, using the Stirling\textquoteright s formula we conclude 
that the number of sequences $a_1, a_2,\ldots, a_n$ of
0\textquoteright s and 1\textquoteright s  such that $\sumaupla{n}<
\de n$ is less or equal than $\exp (n\psi_3(\de))$ with
$\psi_3(\de)\ra 0$ when $\de\ra 0$. Hence, by (\ref{A_n<Cfinal}) and
(\ref{sumofY_n}), we have that whenever $\ga$ and $\ep$ satisfy
(\ref{decompositionYn}),
\begin{equation*}
\# A\pcuad_n(\de)\leq \exp (n\: \psi_4(l,\ga,\de))
\end{equation*}
where
\begin{equation*}
\psi_4(l,\ga,\de)=3(\de+\ga)\log (3(p+1))+ 
2\left(\de+\ga+\frac{1}{l}\right)\log (2l)+ \psi_1(\ga)+\psi_3(\de) .
\end{equation*}
Hence, we have to choose $l$ such that
\begin{equation} \label{choosel}
\frac{2}{l}\log (2l)<\frac{\la}{14}
\end{equation}
and, let $\ga>0$ be such that
\begin{equation} \label{choosega}
\left.
\begin{split}
2\ga\log (2l)&<\frac{\la}{14} \\
3\ga \log (3(p+1))&<\frac{\la}{14} \\
\psi_1(\ga)&<\frac{\la}{14}
\end{split}
\:\: \right\}.
\end{equation}
Next, we find $\ep>0$, using Lemma \ref{lemafrequence}. 
Finally, given $\ep$ and $l$, let $\de>0$ be the constant given by
Lemma \ref{uniformcontinuity} and satisfying 
\begin{equation} \label{choosede}
\left.
\begin{split}
2\de \log (2l)&<\frac{\la}{14} \\
3\de\log (3(p+1))&<\frac{\la}{14} \\
\psi_3(\de)&<\frac{\la}{14}
\end{split}
\:\: \right\} .
\end{equation}
With this choice, $\psi_4(l,\ga,\de)\leq \frac{\la}{2}$. 
Hence the first part of Lemma \ref{lemaprin} is proved, assuming the
three claims. Now we will prove the claims. 
\end{proof}
\end{subsection}

\begin{subsection}{Proof of claims, Lemmas \ref{A_nY_n} and
\ref{lemaprin}}

\noindent
\\
\emph{Proof of Claim \ref{3(p+1)timesthelast}}. Let $I$ be a 
connected component of $C_{\de}(a_1,\ldots,a_s)$.

\emph{Case 1}. $f^s(I)\cap \crit=\emptyset$. In this case, $I$ is
divided at most in 3 connected components of
$C_{\de}(a_1,\ldots,a_s,0)\cup C_{\de}( a_1,\ldots,a_s,1)$. Indeed,
since $I\subset T_{s+1}(x)$ for every $x\in I$, if $I\p\subset I$ is a
component of $C_{\de}(a_1,\ldots,a_s,0)$, it can not exist one
component of $C_{\de}(a_1,\ldots,a_s,1)$ at each side of $I\p$. Hence,
the following situations can occur: 
\begin{itemize}
\item[$i)$] There are two components of $C_{\de}(a_1,\ldots,a_s,0)$ in
$I$, each of them has one extreme of $I$, and between them there is
one component of $C_{\de}(a_1,\ldots,a_s,1)$.
\item[$ii)$] There is exactly one component of 
$C_{\de}(a_1,\ldots,a_s,0)$ in $I$. In this case there is at most one
component of $C_{\de}(a_1,\ldots,a_s,1)$ in $I$.
\item[$iii)$] There are no components of $C_{\de}(a_1,\ldots,a_s,0)$ 
in $I$. In this case $I$ is a component of
$C_{\de}(a_1,\ldots,a_s,1)$.
\end{itemize}

\emph{Case 2}. $f^s(I)\cap\crit\neq\emptyset$. First $I$ is divided at
most in $p+1$ components, each one with at least one boundary which
goes by $f^s$ to $\crit$. After that, following the same arguments
used in case 1, we conclude that each one of these components is
divided at most in 3 components. \\

\noindent
\emph{Proof of Claim \ref{withoutsingularity}}. The proof will 
be by induction on $i$. For $i=1$, it follows by the proof of Claim
\ref{3(p+1)timesthelast}. Let us assume that the statement is true for
$j\leq i-1$. Let $I_1, \ldots, I_t$ be the components of
$C_{\de}(a_1,\ldots,a_s,0^{(i-1)+1})$ contained in $I$. By the
induction hypothesis $t\leq i$ and we assume that
$f^{i}(I)\cap\crit=\emptyset$. We claim that there exist at most one
$k\in\{1,\ldots,t\}$ such that $I_k$ is divided in two components of
$C_{\de}(a_1,\ldots,a_s,0^{i+1})$ (the others $I_k$\s s generate one
or none component of $C_{\de}(a_1,\ldots,a_s,0^{i+1})$). Indeed, if
$I_{k_1}$ and $I_{k_2}$ are divided in two components of
$C_{\de}(a_1,\ldots,a_s,0^{i+1})$, let $I^+_{k_1}$ and $I^-_{k_1}$ be
the components of $C_{\de}(a_1,\ldots,a_s,0^{i+1})$ and let $J_{k_1}$
be the component of $C_{\de}(a_1,\ldots,a_s,0^i,1)$ contained on
$I_{k_1}$. Analogously, let $I^+_{k_2}$, $I^-_{k_2}$, $J_{k_2}$ be the
corresponding components for $I_{k_2}$. Two of the $I^*_{k_j}$
($j\in\{1,2\}, *\in\{+,-\}$) are between $J_{k_1}$ and $J_{k_2}$. This
is a contradiction because $r_{s+i+1}(x)<\de$ for $x\in I^*_{k_j}$ and
$r_{s+i+1}(x)\geq\de$ for $x\in J_{k_1}\cup J_{k_2}$. Hence, there are
at most $i+1$ components of $C_{\de}(a_1,\ldots,a_s,0^{i+1})$
contained in $J$.\qed \\
\noindent
\\
\emph{Proof of Claim \ref{withsingularity}}. 
Let $I$ be a connected component of $C_{\de}(a_1,\ldots,a_s,0)$. 
Then we have $|f^{s+1}(I)|\leq 2\de$, and by Lemma
\ref{uniformcontinuity}, $|f^{s+i}(I)|< \ep$ for $i\leq l+1$. If
$f^{s+j}(I)\cap\crit\neq\emptyset$ for some $j\leq i$, then for all
$x\in I$, $f^{s+j}(x)\in V_{\ep}\crit$. Since $\{s+1,\ldots,s+i\}\cap
\{t_1,\ldots, t_m\}=\emptyset$, then $I\cap
Y_{n,\ep}(t_1,\ldots,t_m)=\emptyset$.

Hence, if $I\cap Y_{n,\ep}(t_1,\ldots,t_m)\neq\emptyset$ and 
$\{s+1,\ldots,s+i\}\cap \{t_1,\ldots, t_l\}=\emptyset$, then
$f^{s+j}(I)\cap\crit=\emptyset$ for all $1\leq j\leq i$. The result
follows using Claim \ref{withoutsingularity}.
\qed
\\
\noindent
\begin{proof}[End of proof of Lemma \ref{lemaprin}]
We have proved the existence of $\de$ (given $\la$) such that 
the number of connected components of $A_n\pcuad(\de)$ is less than
$\exp(n\la/2)$. On the other hand, observe that the choice of $\de$ is
given fundamentally by Lemmas \ref{lemafrequence} and
\ref{uniformcontinuity}. Namely, $\de$ depends on: the constant $\la$
in the definition of $Y_n(\la)$; the uniformity of $\ep$ (given
$\zeta>0$) on the equation (\ref{byequicontinuity}); the uniform
boundedness of $|Df_k|$ on the proof of Lemma \ref{lemafrequence}; the
uniformity of $\de$ (given $\ep$ and $l$) on the equation
(\ref{equationuniformcontinuity}); and the uniform boundedness of the
number of critical points for $f_k$, where $k\geq 0$. So, $\de$
depends only on the modulus of continuity
(\ref{uniformlyequicontinuous}), the uniform bound $\Gamma$ in
(\ref{uniformlybounded}) and the uniform bound $p$ for the cardinal of
the set of critical points, as stated. This concludes the proof of
Lemma \ref{lemaprin}.
\end{proof}

Finally we will prove that Lemma \ref{A_nY_n} follows as a 
consequence of Lemma \ref{lemaprin}.
\begin{proof}[Proof of Lemma \ref{A_nY_n}]
Note that if $J\pcuad$ is a connected component of $A_n\pcuad(\de)$ 
then $f^n$ restricted to $J\pcuad$ is a diffeomorphism onto its
image. Since the set $Y_n(\la)$ is an open subset of $I_0$, there
exist at most countably many components $\{I_k\}_{k\in\N}$ of
$Y_n(\la)\cap A_n(\de)$ on $J\pcuad$. For all $k\in\N$,
\begin{equation*}
|I_k|<(\exp (-n\la)) |f^n(I_k)|,
\end{equation*}
since for every $x\in I_k$, $|Df^n(x)|>\exp(\la n)$. Adding 
these inequalities ($k\in \N)$,
\begin{equation*}
|\cup_{k} I_k|<(\exp (-n\la)) \sum_{k} |f^n(I_k)|
\leq (\exp (-n\la))|f^n(J\pcuad)|.
\end{equation*}
Then, since $|f^n(J\pcuad)|$ is bounded by $|I_0|$, 
\begin{equation*}
|(A_n(\de)\cap J\pcuad)\cap Y_n(\la)|< |I_0| \exp (-n\la)
\end{equation*}
for every connected component $J\pcuad$ of $A_n\pcuad(\de)$. 
To finish the proof of this lemma it is enough to use the estimate of
the number of components of $A_n\pcuad(\de)$ given by Lemma
\ref{lemaprin}. The statement about the dependence of $\de$ follows
from the analogous conclusion on Lemma \ref{lemaprin}.
\end{proof}
\end{subsection}
\end{section}

\begin{section}{Consequences of Theorem \ref{PrincipalB}}
\label{Consequences}
We prove Corollaries \ref{MainCor} and \ref{KellerwithoutSf}. 
Recall that this last result deals with only one single interval map.

\begin{subsection}{Proof of Corollary \ref{MainCor}}
Since (\ref{limsup>2la}) holds for all $x\in H$, $H\subset \cup_{k\geq
n} Y_n(\la)$ (for any $n\in \N$). Thus 
\begin{multline*}
\left|\left(\bigcap_{k\geq n} A_k(\de)\cup 
\complement Y_k(\la)\right)\cap H\right|\leq
\left|\left(\bigcap_{k\geq n} A_k(\de)\cup \complement
(Y_k(\la))\right)\cap \bigcup_{k\geq n} Y_k(\la)\right| \leq
\sum_{k=n}^\infty |A_k(\de)\cap Y_k(\la)|
\end{multline*}
for any $n\in\N$. By Lemma \ref{A_nY_n}, for any $\ep>0$, 
the last sum is less than $\ep$ if $n\geq N(\ep)$. This implies that  
$|(\cap_{n\geq N(\ep)} \cup_{k=n}^{\infty} \complement A_k(\de)\cap
Y_k(\de))\cap H|\geq |H|-\ep$. This means that the set
\begin{equation*} 
\left\{x\in H; \slim \frac{1}{n} \soma r_i(x)\geq\de^2\right\} 
\end{equation*}
has Lebesgue measure greater than  $|H|-\ep$. Since this can be 
done for any $\ep>0$, the corollary follows with $\constantB=\de^2$.
\qed
\end{subsection}

\begin{subsection}{Proof of Corollary \ref{KellerwithoutSf}}
The proof that we give is similar to the proof by de Melo and van
Strien \cite[Theorem V.3.2]{dMvS} for Keller\textquoteright s theorem.
We also construct a Markov map $F$ induced by $f$. 
\begin{proof}[Proof of Corollary \ref{KellerwithoutSf}]
By Theorem \ref{C^3maps} (item $(iii)$) and Corollary 
\ref{MainCor} (applied to $f_n=f$ for $n\geq 0$),
\begin{equation*} 
X=\left\{x\in I_0;\: \limsup_{n\ra+\infty} 
\frac{1}{n}\soma r_i(x)\geq \constantB\right\}
\end{equation*}
has full Lebesgue measure for some $\constantB>0$.

Let us consider a partition $\mathcal{P}$ of $I_0$ into 
(a finite number of) subintervals, with norm less than $\constantB/4$
and such that the set of extremes of such subintervals is forward
invariant. The existence of this partition follows from Theorem
\ref{C^3maps} (items $(i)$ and $(ii)$). Let $\constantB\p$ be the
minimum of the lengths of the elements of $\mathcal{P}$. For every
$x\in I_0$, we denote by $J(x)$ the subinterval of the partition which
contains $x$. And for every $J\in\mathcal{P}$, let us denote by $J^-$
(resp. $J^+$) the rightmost (resp. leftmost) subinterval of the
partition next to $J$. We choose $N\in\N$ such that the intervals of
monotonicity of $f^n$ have length less than $\constantB\p/4$,
for $n\geq N$.

Given $x\in X$, there are infinitely many $k\s s$ such that 
$r_k(x)> \constantB/2$. Let $k(x)\geq N$ be minimal such that
\begin{equation} \label{intervalandneighbours}
f^{k(x)}(T_{k(x)}(x))\supset J(f^{k(x)}(x))
\cup J(f^{k(x)}(x))^{+} \cup J(f^{k(x)}(x))^{-} ,
\end{equation}
and consider $I(x)\subset T_{k(x)}(x)$ such that 
$f^{k(x)}(I(x))= J(f^{k(x)}(x))$. Obviously, for every $y\in I(x)$,
$k(y)\leq k(x)$; and using the forward invariance of the set of
extremes of the subintervals of $\mathcal{P}$, we conclude that in
fact, $k(y)=k(x)$ and $I(y)=I(x)$. Hence, we can define
$F:\cup_{x\in X} I(x)\ra \cup_{J\in\mathcal{P}} J$,
by $F_{|I(x)}={f^{k(x)}}_{|I(x)}$. We claim that this map is Markov
(recall Definition \ref{Markovmap}). Indeed, $(M_3)$ is satisfied
because $|F(I(x))|=|J(F(x))|\geq \constantB\p$. Since $I(x)$ does not
contain extremes of subintervals of $\mathcal{P}$ in its interior,
$I(x)$ is completely contained on some element of $\mathcal{P}$. This
implies that $(M_2)$ holds.

By Theorem \ref{TheoremKoz}, $B(f^{k(x)},T,M)\geq K\p$ for any 
$M\subset T\subset T_{k(x)}$. On the other hand, by
(\ref{intervalandneighbours}), $f^{k(x)}(T_{k(x)}(x))$ contains a
neighborhood $\tau$-scaled of $f^{k(x)}(I(x))$, where
$\tau=4\constantB\p/\constantB$.  Hence, by Koebe Principle (see
\cite[Theorem IV.1.2]{dMvS}), $F$ has bounded distortion on $I(x)$.
It remains to show bounded distortion for the iterates of $F$. Given
$x\in X$ and $s\in \N$; let $m(s,x)\in\N$ be such that
$F^s(x)=f^{m(s,x)}(x)$ and let $I_s(x)$ be the domain of $F^s$
containing $x$. By the choice of $N$, since $m(s,x)\geq N$,
$T_{m(s,x)}(x)$ is contained in at most two elements of $\mathcal{P}$.
Using this and (\ref{intervalandneighbours}) we can prove inductively
that for $x\in X$ and $s\geq 1$ ,
\begin{equation*} 
f^{m(s,x)}(T_{m(s,x)}(x))\supset J(f^{m(s,x)}(x))
\cup J(f^{m(s,x)}(x))^{+} \cup J(f^{m(s,x)}(x))^{-} .
\end{equation*}
So, $(M_1)$ holds and $F$ is a Markov map as we claimed. Hence, there
exists an ergodic absolutely continuous invariant measure $\nu$ for
$F$ (see \cite[Theorem V.2.2]{dMvS}). 
This measure induces an absolutely continuous invariant measure for 
$f$ if $\sum_{i=1}^{\infty} k(i)\nu(I_i)<\infty$ (see \cite[Lemma
V.3.1]{dMvS}). Assume by contradiction that $\sum_{i=1}^{\infty}
k(i)\nu(I_i)=\infty$. By Birkhoff\s s Ergodic~Theorem,
\begin{equation*}
\frac{n_s(x)}{s}=\frac{k(x)+k(F(x))+\ldots+k(F^s(x))}{s}\ra 
\int k(x) d\nu(x)=\sum_{i=1}^{\infty} k(i)\nu(I_i)=\infty
\end{equation*}
for $\nu$-almost every point $x$. For every $x\in X$ and $i\in\N$, 
if $n_i(x)\leq n < n_{i+1}(x)$ and $r_n(x)> \constantB/2$, then
$n-n_i(x)< N$, since in this case $f^n(T_n(x))$ covers one element of
the partition and its two neighbors. Thus we have for $n_{s}(x)\leq n
< n_{s+1}(x)$,
\begin{equation*}
\frac{1}{n}\soma r_i(x) =\frac{1}{n}\sum_{i, r_i(x)> \constantB/2} 
r_i(x)+\frac{1}{n}\sum_{i, r_i(x)\leq \constantB/2} r_i(x)<
\frac{N(s+2)}{n_s(x)}|I_0|+\constantB/2
\end{equation*}
which implies that $\slim 1/n \soma r_i(x)<\constantB$. Since 
it holds for $\nu$-almost every $x$, it contradicts that $X$ has full
Lebesgue measure. Hence there exists
absolutely continuous invariant measure~for~$f$.
\end{proof}
\end{subsection}
\end{section}

\begin{section}{Hyperbolic-like times} \label{Hyperbolic-like}
In this section we develop some preparatory tools for the proof 
of Theorem \ref{PrincipalA}. The arguments are independent from the
previous sections. We prove a similar behavior of points with $r_k\geq
\secondelta$ (for some $\secondelta>0$) and points with $k$ being one
of its $(\sigma\p,\de)$-hyperbolic times. See Lemma 5.2 of \cite{ABV}
and Proposition \ref{difeoboundist} below. Because of this, if
$r_k(z)\geq \secondelta$, we say $k$ is a
$\secondelta$-\emph{hyperbolic-like time} for $z\in M$. 
We need to adapt some notations from subsection \ref{statement2} to
the setting defined by Theorem \ref{PrincipalA}. 

For every $z=(\te,x)\in\toro\times I_0$, let us denote by $T_i(\te,x)$
(or $T_i(z)$) the function $T_i\left(\{f_n\},x\right)$ defined on
subsection \ref{statement2}, considering the sequence
$\{f_{n}\}_{n\geq 0}$ given by $f_n=f_{g^n(\te)}$ for all $n\geq 0$.
We proceed analogously for $L_i(\te,x)$ (or $L_i(z)$), $R_i(\te,x)$
(or $R_i(z)$) and $r_i(\te,x)$ (or $r_i(z)$). We also define
\begin{align*}
\mathcal{T}_i(z):= &\:\{\te\}\times T_i(z);\\
\mathcal{L}_i(z),\, \mathcal{R}_i(z):=&\: \{\te\} \times L_i(z),\,
\{\te\}\times R_i(z);
\end{align*}
for every $z=(\te,x)\in \toro\times I_0$ and every $i\in\N$. In all
the results below we assume that we are in the conditions of Theorem
\ref{PrincipalA}.
\begin{subsection}{Horizontal behavior of dominated skew-products} 
\label{horizontalbehavior}
One important property of our mappings due to the domination
condition is the preservation of the nearly horizontal curves. This
means that the iterates of nearly horizontal curves are still nearly
horizontal. We state it in a precise way. 
\begin{Def} \label{tcurve}
We call $\widehat{X}\subset \toro\times I_0$ a $t-$\emph{curve} if 
there exists $J\subset \toro$  and $X:J\ra I_0$ such that:
$\widehat{X}=\graph(X)$, $X$ is $C^1$ and $|X^{\prime}(\te)|\leq t$
for every $\te\in J$.
\end{Def}
There exists an analogous definition given by Viana (see \cite{Via1},
section 2.1), but he also asks the second derivative to be less than
$t$. He calls the curves with these properties \emph{admissible
curves}. In his setting he proves that the admissible curves are
preserved under iteration. 
\begin{Pro} \label{admissiblecurves}
There exist $\al>0$ and $n_0\in\N$ such that, if $\widehat{X}$ 
is an $\al$-curve and $\fhi^n(\widehat{X})$ is the graph of a $C^1$
map, then $\fhi^n(\widehat{X})$ is an $\al$-curve, provided that
$n\geq n_0$. Moreover, there exists $C_1=C_1(\al)$ such that if
$\widehat{X}$ is a $\al$-curve, then $\fhi^n(\widehat{X})$ is a
$C_1$-curve, for all $n$, provided that $\fhi^n(\widehat{X})$ is a
graph.
\end{Pro}
\begin{proof}
Let $\widehat{X}=\{(\te, X(\te)); \te\in J\}$ be a $C^1$ curve 
with $|X\p(\te)|\leq \al$ for every $\te\in J$. Let us define
inductively for $n\geq 1$, $X_n(g^n(\te))=f(g^{n-1}(\te),
X_{n-1}(g^{n-1}(\te)))$, where $X_0=X$. 
Thus we can prove by induction that
$\fhi^n(\te,X(\te))=(g^n(\te), X_n(g^n(\te)))$, for $n\geq 1$. 

Proceeding similarly as in \cite[Lemma 2.1]{Via1}, using the partial
hyperbolicity (see inequality (\ref{decomdomin1})) and considering
$L=\sup (\partial_\te f/\partial_\te g)$, we have that 
\begin{equation*}
|X_n\p(g^n(\te))|\leq L+\sum_{k=1}^{n-1} LC(\constantdomi)^k 
+C\constantdomi^n\al\leq LCA + C\constantdomi^n \al
\end{equation*}
for $n\geq 1$, where $A=\sum_{k=0}^{\infty} \constantdomi^k$. 
Hence, for some $\al$ and $n_0$ big enough, $|X_n\p(g^n(\te))|\leq
\al$ for all $n\geq n_0$.
\end{proof}
Since all the iterates of $\al$-curves are almost horizontal 
then their lengths are given basically by the derivative of $\fhi$ in
the horizontal direction. We state this in the following result.
\begin{Pro}  \label{contractionadmissible}
Let $C_1=C_1(\al)$ be the constant found on Proposition 
\ref{admissiblecurves}. There exists $C_2=C_2(\al)>0$, such that if
$\widehat{X}=\{(\te, X(\te)); \te\in J\}$ and
$\fhi^k(\widehat{X})=\{(\te, X_k(\te)); \te\in J_k\}$ are graphs with
$|X\p|, |X\p_k|\leq C_1$, then for all $z,w\in \fhi^k(\widehat{X})$,
\begin{equation*} 
\dist_{\widehat{X}}(\fhi^{-k}(z),\fhi^{-k}(w))\leq C_2|
\partial_\te(g^k (\te_k)|^{-1}\dist_{\fhi^k(\hat{X})}(z,w) 
\end{equation*}
for some $\te_k\in J$, where $\dist_A$ is the distance induced by 
the metric over the curve $A$.
\end{Pro}
\begin{proof}
Let us consider the canonical norm in the tangent space, 
i.e, $||(v_1,v_2)||=(|v_1|^2+|v_2|^2)^{\frac{1}{2}}$,
where $v=(v_1,v_2)\in T_z(\toro\times I_0), v_1\in T_\te \toro, v_2\in
T_x I_0$ and $z=(\te, x)$. 

We denote the tangent vector to the curve $\widehat{X}$ at the 
point $(\te, X(\te))$ by $(v_1(\te),v_2(\te))$. Let us consider
$\te_z, \te_w\in J$ such that $\fhi^k(\te_z, X(\te_z))=z$ and
analogously for $w$. Then, since $|v_2(\te)|/|v_1(\te)|\leq C_1$,
\begin{align*}
\dist_{\fhi^k(\widehat{X})}(z,w)&=\int_{\te_z}^{\te_w} 
||D\fhi^k(\te, X(\te))(v_1(\te), v_2(\te))|| d\te
\geq \int_{\te_z}^{\te_w} |\partial_\te g^k(\te)||v_1(\te)| d\te 
\geq\qquad\qquad\\
&\geq \frac{1}{(1+(C_1)^2)^{\frac{1}{2}}} \int_{\te_z}^{\te_w} 
|\partial_\te g^k(\te)|(|v_1(\te)|^2+|v_2(\te)|^2)^{\frac{1}{2}} d\te
\\
&\geq \frac{1}{(1+(C_1)^2)^{\frac{1}{2}}} |\partial_\te g^k(\te_k)| 
\dist_{\widehat{X}}(\fhi^{-k}(z),\fhi^{-k}(w))
\end{align*}
where $\te_k$ is such that $|\partial_\te g^k(\te_k)|\leq 
|\partial_\te g^k(\te)|$ for $\te\in [\te_z,\te_w]$. This means that
we may take $C_2=(1+(C_1)^2)^{\frac{1}{2}}$.
\end{proof}
\end{subsection}

\begin{subsection}{Properties of the hyperbolic-like times}
In the case that $k$ is a hyperbolic time for $z$, there is
contraction for all the inverse iterates in a certain neighborhood of
$\fhi^k(z)$. In the case of hyperbolic-like times this property is not
necessarily verified. However, it holds the following result

\begin{Pro} \label{difeoboundist}
Given $\secondelta>0$, there exists $\delta_1>0$ such that for 
$z\in M$ with $r_k(z)\geq \secondelta$ for some $k\in\N$, there exists
a neighborhood $V_k(z)$ of $z$ such that $\fhi^k:V_k(z)\ra
B_{\delta_1}(\fhi^k(z))$ is a diffeomorphism with bounded distortion
(it depends on $\secondelta$, but it is independent of $z$ and $k$).
\end{Pro}
\begin{proof}
Let $z=(\te,x)\in\toro\times I_0$ for some $\te\in\toro$ and $x\in
I_0$.
Let $\mathcal{T}_k(z)$ be the maximal interval such that 
$\fhi^j(\mathcal{T}_k(z))\cap\crit=\emptyset$ for all $j<k$ and let
$\mathcal{L}_k(z)$, $\mathcal{R}_k(z)$ be the components of
$\mathcal{T}_k(z)\setminus\{z\}$. By hypothesis
$|\fhi^{k}(\mathcal{L}_k(z))|\geq \secondelta$ and
$|\fhi^{k}(\mathcal{R}_k(z))|\geq\secondelta$. Let us consider
$\mathcal{I}_k(z)\subset \mathcal{T}_k(z)$ such that every component
of $\fhi^k(\mathcal{T}_k(z))\setminus \fhi^k(\mathcal{I}_k(z))$ has
length equal to $\secondelta/2$. In particular, we have that both
components of $\fhi^{k}(\mathcal{I}_k(z)\setminus \{z\})$ have length
greater or equal than $\secondelta/2$. 
By definition of $\fhi$, we know that the horizontal component of 
$\fhi^k(z)$ is $g^k(\te)$. Let us consider $\eta_1>0$ and $\eta_2>0$
such that $g^k:(\te-\eta_1,\te+\eta_2)\ra
(g^k(\te)-\rho\p,g^k(\te)+\rho\p)$ is a diffeomorphism. Here $\rho\p$
is a sufficiently small constant whose value will be made precise in
(\ref{constantrhop}). 

Let $I_k(z)$ be the projection of $\mathcal{I}_k(z)$ onto $I_0$. Let
us consider the set $B_k(z)=(\te-\eta_1,\te-\eta_2)\times I_k(z)$. For
every $w=(\te, x_w)\in \mathcal{I}_k(z)$, we denote by
$\mathcal{B}_{w}$ the line joining the points $(\te-\eta_1,x_w)$ and
$(\te+\eta_2,x_w)$. We denote by $\mathcal{B}^j_{w}$ (for $j\leq k$)
the curve given by the image of $\mathcal{B}_{w}$ under $\fhi^j$, i.e,
which satisfies $\fhi^{j}(\mathcal{B}_{w})=\mathcal{B}^j_{w}$. Observe
that $\mathcal{B}^0_{w}=\mathcal{B}_{w}$ for any $w\in I_k(z)$. 

In the same way we denote by $w^k$ the image under $\fhi^k$ of the
point $w=w^0$ and by $\mathcal{T}^j$ the set
$\fhi^j(\mathcal{T}_k(z))$ (since $z$ and $k$ are fixed along the
proof, there is no confusion in omitting in the notation the
dependence of $\mathcal{T}^j$ on $z$ and $k$). 

\begin{Cla} \label{difeofhi^k}
$\fhi^k: B_k(z)\ra \fhi^{k}(B_k(z))$ is a diffeomorphism.
\end{Cla}
\begin{proof}
We will use the bounded distortion of the map $g$. Namely, 
there exists $D>0$ such that, if we have $J\subset\toro$ and $n\in\N$
for which $g^n:J\ra g^n(J)$ is a diffeomorphism, then
\begin{equation} \label{distortionofg}
\frac{|\partial_\te g^n(\te)|}{|\partial_\te g^n(\omega)|}\leq D
\end{equation}
for all $\te,\omega\in J$. We claim that 
$\mathcal{B}^j_{w}\cap \crit=\emptyset$ for $j< k$ and for any $w\in
\mathcal{I}_k(z)$. 

Recall the constants $C$, $C_1$, $C_2$ and $D$, specified in 
(\ref{decomdomin1}), Proposition \ref{admissiblecurves}, Proposition
\ref{contractionadmissible} and (\ref{distortionofg}), respectively.
Let us assume that for every $w\in \mathcal{I}_k(z)$,
$|\mathcal{B}^{k}_{w}|\leq \rho$, where $\rho$ satisfies the
conditions
\begin{equation}\label{constantrho}
 C_2 \rho < (\secondelta/4)(DC)^{-1} \quad \text{ and } \quad
 \rho C_1 <\secondelta/4 .
\end{equation}
Let us fix $w\in \mathcal{I}_k(z)$. First, for all $j\leq k$, 
$\mathcal{B}^j_{w}$ are $C_1$-curves (see Definition \ref{tcurve} and
Proposition \ref{admissiblecurves}). On the other hand, there exists
$C_2$ such that $|\mathcal{B}^{k-j}_{w}|\leq C_2 |\partial_\te
g^j(\te_j)|^{-1} |\mathcal{B}^{k}_{w}|$ for some $(\te_j, x_j)\in
\mathcal{B}^{k-j}_{w}$, where $|\mathcal{B}|$ denotes the arc length
of the  curve $\mathcal{B}$ (see Proposition
\ref{contractionadmissible}).

For $1\leq j\leq k$, let us denote by $\mathcal{I}_{w,+}^{k-j}$ and
$\mathcal{I}_{w,-}^{k-j}$ the connected components of
$\mathcal{T}^{k-j}\setminus \{w^{k-j}\}$.  
By the mean value theorem, we have that 
$|\mathcal{I}_{w,+}^{k-j}|\geq (\prod_{i=0}^{j-1} |\partial_x
f(\fhi^i({\om}_j,y_j))|)^{-1}(\secondelta/2)$ for some $(\om_j,y_j)\in
\mathcal{I}_w^{k-j}$; and $|\mathcal{I}_{w,-}^{k-j}|\geq
(\prod_{i=0}^{j-1} |\partial_x
f(\fhi^i(\om_j\p,y_j\p))|)^{-1}(\secondelta/2)$, for some
$(\om_j\p,y_j\p)\in \mathcal{I}_{w,-}^{k-j}$. So, two cases can occur:
$(i)$ $|\mathcal{I}_{w,+}^{k-j}|\leq|\mathcal{I}_{w,-}^{k-j}|$, 
or $(ii)$ $|\mathcal{I}_{w,+}^{k-j}|>|\mathcal{I}_{w,-}^{k-j}|$.

Let us assume that we have the case $(i)$ 
(the other case is totally analogous). Then combining
(\ref{decomdomin1}) and (\ref{distortionofg}), we have
\begin{equation*}
|\partial_\te g^j(\te_j)|^{-1}<D C\,\constantdomi^j \,
\left(\prod_{i=0}^{j-1} \left|\partial_x
f(\fhi^i(\om_j,y_j))\right|\right)^{-1}\hspace{-0.25cm} .
\end{equation*}
From Proposition \ref{contractionadmissible}, the last inequality 
and (\ref{constantrho}), we have for $1\leq j\leq k$,
\begin{equation} \label{admissiblevsvertical1}
|\mathcal{B}^{k-j}_{w}|\leq C_2\, |\partial_\te g^j(\te_j)|^{-1}\, 
\rho < \constantdomi^j\,\left(\prod_{i=0}^{j-1} \left|\partial_x
f(\fhi^i(\om_j,y_j))\right|\right)^{-1}(\secondelta/4) \leq
\constantdomi^j \, \frac{\distv(w^{k-j},\crit)}{2} .
\end{equation}
for $w\in \mathcal{I}_k(z)$. This equation, and the condition 
$(F_2)$ satisfied by the skew-product, implies that
$\mathcal{B}^{k-j}_{w}\cap\crit=\emptyset$ (for every $1\leq j\leq
k$). Therefore the map  $\fhi^k:B_k\ra \fhi^k(B_k)$ is a local
diffeomorphism.

We claim that the map is injective. Indeed, if there exist 
$(\te_1, x_1)$ and $(\te_2, x_2)$ in $B_k$ such that
$\fhi^k(\te_1,x_1)=\fhi^k(\te_2,x_2)\in B$, since in the horizontal
direction there is expansion ($\partial_\te g> 1$), it must be
$\te_1=\te_2$. Next, by the differentiability of the functions
$f(\te,\cdot)$, if $x_1\neq x_2$, there must be at least one point
$(\te_1,x_w)$ between $(\te_1, x_1)$ and $(\te_1, x_2)$ and $j<k$ such
that this point is mapped by $\fhi^j$ in a critical point.  But this
would imply that $\mathcal{B}^j_{w}\cap\crit\neq \emptyset$ (for some
$w\in \mathcal{I}_k(z)$), which is a contradiction. Hence $x_1=x_2$,
which implies that the map $\fhi^k:B_k\ra \fhi^k(B_k)$ is injective.

Therefore, if $\rho$ is as in (\ref{constantrho}), Claim 
\ref{difeofhi^k} follows. It just remains to state precisely the value
of $\rho\p$. Given $\rho$, we choose $\rho\p<\rho$ maximal such that
\begin{align} \label{constantrhop}
\text{given } J\subset\toro \text{ interval with length } \rho\p 
\text{ and } X:J\ra I_0 \text{ a curve with } |X\p|\leq C_1, \\
 \text{ the arc length of } \graph(X) 
\text{ is less or equal than } \rho.\qquad\qquad \nonumber
\end{align}
where $C_1$ is the constant given in Proposition
\ref{admissiblecurves}. It finishes the proof of the claim.

\end{proof}
Let us prove now that the transformation of Claim \ref{difeofhi^k} has
bounded distortion.
\begin{Cla} \label{verticaldistortion}
There exists $K_1=K_1(\secondelta)>0$ such that for 
$z_1, z_2\in \mathcal{I}_k(z)\subset B_k(z)$,
\begin{equation*} \label{distortionvertical}
\frac{1}{K_1}\leq \frac{|\det D\fhi^k(z_1)|}{|\det D\fhi^k(z_2)|}\leq
K_1.
\end{equation*}
\end{Cla}
\begin{proof}
Let $z_1$ and $z_2$ be points in $\mathcal{I}_k(z)$, where
$z=(\te,x)$ for some $x\in I_0$ and $\te\in\toro$. We have that
$I_k(z)\subset T_k(z)$ (since $\mathcal{I}_k(z)\subset
\mathcal{T}_k(z)$ and these sets are the corresponding projections
onto $I_0$). Recall the notation
$f^k_\te=f_{g^{k-1}(\te)}\circ\ldots\circ f_{g(\te)}\circ f_{\te}$,
where $f_\te(x)=f(\te,x)$ for $\te\in \toro$ and $x\in I_0$. 
Since $\fhi^j(\mathcal{T}_k(z))\cap\crit=\emptyset$ for $j<k$, we have
that $f^k_\te:T_k(z)\ra$ $f^k_\te(T_k(z))$ is a $C^3$ diffeomorphism. 
By the way we have chosen $\mathcal{I}_k(z)$ we know that every
component of $f^k_\te(T_k(z))\setminus$ $f^k_\te(I_k(z))$ has length
equal to $\secondelta/2$. Then there exists $\kappa>0$ (depending only
on $\secondelta$), such that $f^k_\te(T_k(z))$ contains a
$\kappa$-scaled neighborhood of $f^k_\te(I_k(z))$ (i.e, both
components of $f^k_\te(T_k(z))\setminus f^k_\te(I_k(z))$ have length
$\geq \kappa |J|$). Thus, by Koebe Principle (see \cite[Theorem
IV.1.2]{dMvS}), there exists $K_1=K_1(\kappa)>0$ such that for $y_1,
y_2\in I_k(z)$,
\begin{equation*}
\frac{1}{K_1}\leq\frac{|Df^k_\te(y_1)|}{|Df^k_\te(y_2)|} \leq K_1 .
\end{equation*}

Now, for $z_1=(\te, y_1)\in \mathcal{I}_k(z)$, 
$|\det D\fhi^k(z_1)|=|\partial_\te g^k(\te)| |Df^k_\te(y_1)|$. It
finishes the proof. 
\end{proof}

\begin{Cla} \label{horizontaldistortion}
There exists $K_2=K_2(\secondelta)>0$ such that for 
$z_1\in \mathcal{I}_k(z)$ and $z_2$ in the same horizontal leaf
$\mathcal{B}_{z_1}$ of $z_1$,
\begin{equation*} 
\frac{1}{K_2}\leq \frac{|\det D\fhi^k(z_1)|}{|\det D\fhi^k(z_2)|}\leq
K_2.
\end{equation*}
\end{Cla}
\begin{proof}
Using the condition $(F_2)$ satisfied by the skew-product, together
with (\ref{distortionofg}), we conclude
\begin{equation*}
\left|\log\frac{|\det D\fhi^k(z_1)|}{|\det
D\fhi^k(z_2)|}\right| \leq \log D+ B\sum_{j=1}^{k}
\frac{\dist(\fhi^{k-j}(z_1),\fhi^{k-j}(z_2))}{\distv(\fhi^{k-j}(z_1),
\crit)^{}}
\end{equation*}
and by (\ref{admissiblevsvertical1}), we have 
\begin{equation*}
\left|\log\frac{|\det D\fhi^k(z_1)|}{|\det
D\fhi^k(z_2)|}\right|\leq \log D+ B\sum_{j=1}^{k} 
\constantdomi^j \leq B\p\sum_{j=1}^{\infty} \constantdomi^j=K_2\p .
\end{equation*} 
This concludes the proof of the claim.
\end{proof}

Combining Claim \ref{verticaldistortion} and Claim 
\ref{horizontaldistortion}, we get that $\fhi^k:B_k(z)\ra
\fhi^k(B_k(z))$ has bounded distortion. To finish the proof of
Proposition \ref{difeoboundist}, it remains to show that $\fhi^k(B_k)$
contains $B_{\delta_1}(\fhi^k(z))$ for some~$\delta_1>0$.

Recall that $z=(\te,x)$. The image of the horizontal curves of 
$B_k(z)$, i.e. $\mathcal{B}_{w}^k$, are $C_1$-curves for all $w\in
\mathcal{I}_k(z)$ (see Definition \ref{tcurve} and Proposition
\ref{admissiblecurves}). Using this fact and (\ref{constantrho}) we
conclude that $\fhi^k(B_k(z))$ contains the set 
\begin{equation*}
(g^k(\te)-\rho\p,g^k(\te)+\rho\p)\times 
(f_{\te}^k(x)-\secondelta/4,f_{\te}^k(x)+\secondelta/4) 
\end{equation*}
where $\rho\p$ was defined on (\ref{constantrhop}) and it 
does neither depend on the point $z$, nor on the iterate $k$. Hence
there exists $\delta_1>0$ such that $B_{\delta_1}(\fhi^k(z))\subset
\fhi^k(B_k(z))$. Considering $V_k(z)\subset B_k(z)$ such that
$\fhi^k(V_k(z))=B_{\delta_1}(\fhi^k(z))$, Proposition
\ref{difeoboundist} follows.
\end{proof}
\end{subsection}

\begin{subsection}{Neighborhoods associated to hyperbolic-like times}
For every $\secondelta>0$ and $i\in\N$, we denote by 
$H_i(\secondelta)$ the set of points $z\in M$ with $r_i(z)\geq
\secondelta$. The following lemma will be very useful in the
construction of the absolutely continuous invariant measure for
$\fhi$.

\begin{Le} \label{tauHi} 
Given $\secondelta>0$, there exists $\tau=\tau(\secondelta)>0$ 
such that for every $i\in\N$ and for any measurable set $Z$, there
exists a finite set of points $z_1,\ldots,z_N$ in $H_i(\secondelta)$
and neighborhoods $V\p_i(z_1),\ldots, V\p_i(z_N)$ which are two-by-two
disjoint. For every $k=1,\ldots,N$,  $\fhi^i:V\p_i(z_k)\to
B_{\delta_1/4}(\fhi^i(z_k))$ is a diffeomorphism with bounded
distortion and the union $W_i=V\p_i(z_1)\cup\ldots\cup V\p_i(z_N)$
satisfies
\begin{equation*}
\leb(W_i\cap H_i(\secondelta)\cap Z)\geq \tau\leb(H_i(\secondelta)\cap
Z) .
\end{equation*}
\end{Le}
\begin{Rem}
The constant $\delta_1$ and the distortion bound which appear in 
this lemma are the same given in Proposition \ref{difeoboundist},
which are independent on the point $z\in M$ and on the iterate
$i\in\N$.
\end{Rem}

\begin{Proof}
This is analogous to the proofs of Proposition 3.3 and Lemma 3.4
of \cite{ABV}, using hyperbolic-like times instead of hyperbolic
times.
\end{Proof}
\end{subsection}
\end{section}

\begin{section}{Absolutely continuous invariant measure}
\label{Absolutely}
Here we prove Theorem \ref{PrincipalA}. In order to do it, 
we need to control the measure of the points with many (positive
density) hyperbolic-like times.

\begin{subsection}{Points with infinitely many hyperbolic-like times}
We are going to show that, for some $\thirdelta>0$, the points with 
many $\thirdelta$-hyperbolic-like times are a positive Lebesgue
measure set. 
 
Recall that we denote $\toro\times I_0$ by $M$ and the 
Lebesgue measure of $M$ by $\leb$. Given any $\la>0$, let $Z(\la)$ be
the set of points in $M$ for which the limit in (\ref{nue}) is greater
than $2\la$. Also, for $n\in \N$, we define,
\begin{equation*}\label{Zn}
Z_n(\la)=\Biggl\{z\in Z(\la); \frac{1}{n}\sum_{j=0}^{n-1}
\log \|D\fhi(\fhi^j(z))^{-1}\|^{-1}>\la\Biggr\},
\end{equation*}
and for $\de>0$,
\begin{equation*} \label{An}
A^M_n(\de)= \Biggl\{ z\in M; \frac{1}{n}\soma r_i(z)<\de^2,
\:\:r_n(z)>0\Biggr\}
\end{equation*}
where $r_i(z)=r_i(\te,x)$ denotes the function 
$r_i\left(\{f_n\},x\right)$ defined on subsection \ref{statement2},
considering the sequence $f_n=f_{g^n(\te)}$ for $n\geq 0$. As we will
now see, these sets have relation with the sets defined by equations
(\ref{Yntheta}) and (\ref{Antheta}). 

We denote by $A_n(\te,\de)$ the set $A_n\left(\{f_n\},\de\right)$ 
(defined on (\ref{Antheta})), and by $Y_n(\te,\la)$ the set
$Y_n\left(\{f_n\},\la\right)$ (defined on (\ref{Yntheta})), with
$f_n=f_{g^n(\te)}$ for $n\geq 0$. Thus, we can conclude that
\begin{equation} \label{totaltoleaf}
Z_n(\la)\subset \cup_{\te\in T} (\te\times Y_n(\te,\la))
\qquad  \text{ and }\qquad
A^M_n(\de)=\cup_{\te\in T} (\te\times A_n(\te,\de)) .
\end{equation}

For every $\te\in\toro$, $\{f_{g^n(\te)}\}$ is a $C^1$-uniformly 
equicontinuous and $C^1$-uniformly bounded sequence of smooth maps.
It also holds that $p=\sup \#\crit_{g^n(\te)}<\infty$. Thus, we are in
the context of Lemma \ref{A_nY_n}. Moreover, for fixed $\la>0$, the
constant $\de$ given by Lemma \ref{A_nY_n} does not depend on $\te$,
i.e., the constant $\de$ is the same for any sequence
$\{f_{g^n(\te)}\}$. This happens because the modulus of continuity
(\ref{uniformlyequicontinuous}), the uniform bound $\Gamma$ in
(\ref{uniformlybounded}) and the uniform bound $p$ for the number of
critical points, are the same for any sequence $\{f_{g^n(\te)}\}$
(varying $\te\in\toro$). The last is true since $\fhi$ is $C^3$ and
$(F_1)$ holds.

\begin{Pro} \label{A_ntotal}
In the conditions of Theorem \ref{PrincipalA}, given $\la>0$, 
there exist $\thirdelta=\thirdelta(\la)>0$ such that
\begin{equation*} \label{goodset}
\leb\Biggl(\Biggl\{z\in Z(\la)\:;\:\soma r_i(z)\geq 2\thirdelta
n, \text{for all } n\geq n_0\Biggr\}\Biggr)\geq \leb\left( \cap_{n\geq
n_0} Z_n(\la)\right)/2 .
\end{equation*}
for $n_0$ big enough. Moreover, for Lebesgue almost every
$z\in Z(\la)$, $\liminf_{n\to\infty} \frac{1}{n} \soma r_i(z)\geq
2\thirdelta $.
\end{Pro}
\begin{proof} 
For $\la,\de>0$ and every $N\in \N$,
\begin{multline*}
\int_{\toro}\int_{I_0} \bigchi_{\{\cap_{n=N}^{\infty}
\complement A^M_n(\de)\cap Z_n(\la)\}}(\te,x)
dm_{I_0}(x)dm_{\toro}(\te)\geq 
\int_{\toro}\int_{I_0} \bigchi_{\{\cap_{n=N}^{\infty}
Z_n(\la)\}}(\te, x)dm_{I_0}(x)
dm_{\toro}(\te)\:-\: \\
-\int_{\toro}\int_{I_0} \bigchi_{\{\cup_{n=N}^{\infty}
A^M_n(\de)\cap Z_n(\la)\}}(\te, x)
dm_{I_0}(x)dm_{\toro}(\te) .
\end{multline*}
where $m_{I_0}$ and $m_{\toro}$ denote the Lebesgue measure 
on $I_0$ and $\toro$. On the other hand, by Lemma \ref{A_nY_n}, there
exists $\de>0$ such that for every $\te\in \toro$,
\begin{equation*}
m_{I_0}\left(\bigcup_{n=N}^{\infty}A_n(\te,\de)\cap
Y_n(\te,\la)\right)\ra 0,
\end{equation*}
when $N\ra\infty$. This together with (\ref{totaltoleaf}) yield, 
\begin{equation*}
\int_{\toro}\int_{I_0} \bigchi_{\{\cup_{n=N}^{\infty}
A^M_n(\de)\cap Z_n(\la)\}}(\te, x)
dm_{I_0}(x)dm_{\toro}(\te)  
\leq \int_{\toro}\int_{I_0} \bigchi_{\{\cup_{n=N}^{\infty}
A_n(\te,\de)\cap
Y_n(\te,\de)\}}(x)dm_{I_0}(x)dm_{\toro}(\te)\longrightarrow 0
\end{equation*}
when $N\rightarrow\infty$. 
Considering $\thirdelta$ such that $2\thirdelta<\de^2$, the
proposition follows.
\end{proof}
\end{subsection}

\begin{subsection}{Positive density of the hyperbolic-like times}
We prove that for every point 
$z$ such that $\soma r_i(z)\geq 2\thirdelta n$, (for some
$\thirdelta>0$),
the density of hyperbolic-like times is uniformly positive. 

Recall that for every $\thirdelta>0$ and $n\in\N$, we denote 
by $H_n(\thirdelta)$ the set of points $z\in M$ with $r_n(z)\geq
\thirdelta$.
\begin{Le} \label{densityofhyptimes}
Given $\thirdelta>0$, there exists $\zeta=\zeta(\thirdelta)>0$ such
that
\begin{equation*}
\frac{\#\:\{\,1\leq i\leq n ; \, z\in H_i(\thirdelta)\,\}}{n}\geq
\zeta
\end{equation*}
for any $z$ such that
$\soma r_i(z)\geq 2\thirdelta n$.
\end{Le}
\begin{proof}
Considering $c_2=2\thirdelta$ and $c_1=\thirdelta$, applying 
the Pliss lemma (see \cite{Pliss}), there are $q\geq \zeta n$ and
$0<n_1<\ldots<n_q\leq n$ such that
\begin{equation*}
\sum_{j=k+1}^{n_i} r_j(z)\geq \thirdelta(n_i-k) \quad
\text{for every}\quad 0\leq k< n_i, \text{ and } i=1,\ldots,q.
\end{equation*}
Observe that $\zeta$ does not depend on $z$ neither on $n$. 
Hence, for any $z$ as in the statement of the lemma, there exist
$0<n_1<\ldots<n_q\leq n$ such that $r_{n_i}(z)\geq \thirdelta$ 
($1\leq i\leq q$) and $q/n\geq\zeta$.  
\end{proof}
\end{subsection}

\begin{subsection}{Construction of the measure}
\label{constructionofmeasure}
We consider the sequence
\begin{equation*}
\mu_n=\frac{1}{n}\soma \fhi^i_*\leb
\end{equation*}
of averages of forward iterates of Lebesgue measure on $M$. 
The main idea is  to decompose $\mu_n$ (for every $n$) as a sum of two
measures, $\nu_n$ and $\eta_n$, such that $\nu_n$ is uniformly
absolutely continuous and has total mass bounded away from zero. The
measure $\nu_n$ will be the part of $\mu_n$ carried on balls of radius
$\delta_1$ around points $\fhi^i(z)$, where $z$ is a point which has
$1\leq i\leq n$ as $\thirdelta$-hyperbolic-like time.

Let us fix $\la>0$ such that $\leb(Z(\la))>0$.  
Let us consider the corresponding $\thirdelta=\thirdelta(\la)>0$ from
Proposition \ref{A_ntotal}. Let $W_i$ be the set given by Lemma
\ref{tauHi} for $\secondelta=\thirdelta$. We consider the measures
\begin{equation*}
\nu_n=\frac{1}{n}\soma \fhi^i_*\leb_{W_i}
\end{equation*}
and $\eta_n=\mu_n-\nu_n$, where $\leb_{X}$ denotes the 
restriction of the Lebesgue measure to $X$. 
\begin{Pro} \label{finaltheoremA}
The measures $\nu_n$ are uniformly absolutely continuous and 
give positive (bounded away from zero) weight to $Z(\la)$, for all
large $n$.
\end{Pro}
\begin{proof}
By Proposition \ref{difeoboundist}, the measures 
$\fhi^i_*\leb_{V_i(z)}$ are absolutely continuous and the densities
are uniformly bounded from above. It also holds for the measures
$\fhi^i_*\leb_{W_i}$, since $W_i$ is a disjoint union of sets $V_i's$.
Therefore, $\nu_n$ are absolutely continuous and the densities are
uniformly bounded from above. It just remains to prove the claim about
$Z(\la)$. By Lemma \ref{tauHi}, there exists $\tau=\tau(\thirdelta)$
such that
\begin{equation*} \label{totalmass} 
\nu_n(Z(\la))\geq \tau \frac{1}{n} \soma \leb(H_i(\thirdelta)\cap
Z(\la)).
\end{equation*}
So, it suffices to control the right side of the last expression. 
For this, let us consider the measure $\pi_n$ in $\{1,2, \ldots,n\}$
defined by $\pi_n(B)=\#(B)/n$, for every subset $B$. Using
Fubini\textquoteright s theorem, we have
\begin{equation*}
\frac{1}{n} \soma \leb(H_i(\thirdelta)\cap Z(\la))=
\int\int_{Z(\la)}\bigchi(z,i)d\leb(z)d\pi_n(i)=
\int_{Z(\la)}\int \bigchi(z,i)d\pi_n(i)d\leb(z)
\end{equation*}
where $\bigchi(z,i)=1$ if $z\in H_i(\thirdelta)$ and 
$\bigchi(z,i)=0$ otherwise. By Lemma \ref{densityofhyptimes}, it holds
$\int \bigchi(z,i)d\pi_n(i)\geq \zeta$ if $z$ is such that $\soma
r_i(z)\geq 2\thirdelta n$. Hence
\begin{equation*}
\frac{1}{n}\: \soma \leb(H_i(\thirdelta)\cap Z(\la))\:\geq \:\zeta
\leb\,\Biggl(\Biggl\{\,z\in Z(\la);\:\soma r_i(z)\geq 2\thirdelta
n\Biggr\} \Biggr).
\end{equation*}
In this way, we conclude using Proposition \ref{A_ntotal} that the
weight of $Z(\la)$ for the measure $\nu_n$ is bounded away from zero,
for $n$ big enough.
\end{proof}

The limit of any convergent subsequence of $\{\nu_n\}_n$ is an
absolutely continuous measure. It just remains to prove that we can
find our measure in such a way that it is invariant. Let us choose
$\{n_k\}_k$ such that $\mu_{n_k}$, $\nu_{n_k}$ and $\eta_{n_k}$
converge to $\mu$, $\nu$ and $\eta$, respectively. 
We can decompose $\eta=\eta^{ac}+\eta^{s}$ as the sum of an 
absolutely continuous measure $\eta^{ac}$ and a singular measure
$\eta^s$ (with respect to Lebesgue measure). Then, 
$\mu=(\nu+\eta^{ac})+\eta^{s}$ gives one decomposition of $\mu$ as sum
of one absolutely continuous and one singular measure. Since the push
forward under $\fhi$ preserves the class of absolutely continuous
measures and $\mu$ is invariant, 
$\mu=\fhi_*\mu=\fhi_*(\nu+\eta^{ac})+\fhi_*\eta^{s}$
gives another decomposition of $\mu$ as sum of one absolutely
continuous and one singular measure. By the uniqueness of the
decomposition we must have $\fhi_*(\nu+\eta^{ac})=\nu+\eta^{ac}$.
Hence, $\nu+\eta^{ac}$ is a non-zero absolutely continuous invariant
measure for $\fhi$.
\end{subsection}

\begin{subsection}{Ergodicity and finite number of measures} 
\label{ergodicityofmeasure}
To finish the proof of Theorem \ref{PrincipalA}, it remains to 
prove the ergodicity of the absolutely continuous invariant measure
and the finiteness claim in the statement of the theorem. Fixed
$\la>0$, we consider the constant $\thirdelta>0$ given on Proposition
\ref{A_ntotal}. Recall that for $\secondelta=\thirdelta$, we denote by
$V_k(z)$ (for $k\in\N, z\in M$) the neighborhood constructed on
Proposition \ref{difeoboundist}: it is mapped diffeomorphically onto
the ball of radius $\delta_1>0$ around $\fhi^k(z)$ by $\fhi^k$.
\begin{Le} \label{containinneigh}
Let $\la>0$ and $\thirdelta=\thirdelta(\la)$ be as in Proposition 
\ref{A_ntotal}. Let us consider $G_0\subset M$ an open set. Then for
any $z\in Z(\la)\cap G_0$, $V_k(z)\subset G_0$ whenever $z\in
H_k(\thirdelta)$ and $k$ is big enough.
\end{Le}
\begin{proof}
In Proposition \ref{difeoboundist} we fixed the constant $\rho\p$
according to (\ref{constantrhop}) and we constructed the neighborhood
$V_k(z)$. This neighborhood is such that $V_k(z)\subset
B_k(z)=(\te-\eta_1,\te-\eta_2)\times I_k(z)$, where: $(i)$
$g^k:(\te-\eta_1,\te+\eta_2)\ra (g^k(\te)-\rho\p,g^k(\te)+\rho\p)$ is
a diffeomorphism; $(ii)$ $I_k(z)\subset T_k(z)$ and $f_\te^k$ is a
diffeomorphism restricted to $T_k(z)$. 
To conclude the proof, it is enough to show that $\eta_1$, $\eta_2$ 
and $|I_k(z)|$ goes to zero when $k$ goes to infinity. The claim about
$\eta_1$ and $\eta_2$ follows from the uniform expansion of $g$. Since
$z\in Z_k(\la)$ for $k$ big enough, the bounded distortion on
$f_\te^k:I_k(z)\to f_\te^k(I_k(z))$ (see the proof of Claim
\ref{verticaldistortion}) implies that $|I_k(z)|$ goes to zero.
\end{proof}
\begin{Le} \label{invariantsets}
For any positively invariant set $G\subset Z(\la)$ there exists some 
disk $\Delta$ with radius $\delta_1/4$ such that $\leb(\Delta\setminus
G)=0$. 
\end{Le}
\begin{proof}
The proof is analogous to the proof of Lemma 5.6 of \cite{ABV}. We
make use of $\thirdelta(\la)$-hyperbolic-like times instead of
$(\sigma,\delta)$-hyperbolic times. Thus, the only difference is the
reason why the neighborhoods $V_k(z)$ decrease with $k$. In our case,
this is given by Lemma \ref{containinneigh}. 
\end{proof}
\begin{proof}[End of proof of Theorem \ref{PrincipalA}]
At the end of subsection \ref{constructionofmeasure}, we construct an
absolutely continuous invariant measure $\nu_0:=\nu+\eta^{ac}$ with
$\nu_0(Z(\la))>0$. Since $Z(\la)$ is positively invariant,
we can suppose that $\nu_0(Z(\la))=1$. 
On the other hand, by Lemma \ref{invariantsets}, each invariant 
set on $Z(\la)$ with positive $\nu_0$-measure has full Lebesgue
measure in some disk with fixed radius. Since the manifold is compact,
there can be only finitely many disjoint invariant sets on $Z(\la)$
with positive $\nu_0$-measure. Hence $\nu_0$ can be decomposed as a
sum of ergodic measures. Namely, $\nu_0=\sum_{i=1}^{l}
\nu_0(D_i)\nu_{i}$, where $D_1, \ldots, D_l$ are disjoint invariant
sets with positive measure and $\nu_{i}$ is the normalized restriction
of $\nu_0$ to $D_i$.  The measures $\nu_i$ ($1\leq i\leq l$) are
ergodic absolutely continuous probabilities. Therefore, they are SRB
measures.

If $Z_1=Z(\la)\setminus \cup_{i=1}^{s} B_{i}$ (where $B_i$ denotes the
basin of the measure $\mu_i$) has positive Lebesgue measure, then we
can repeat the arguments in this section with $Z_1$ in the place of
$Z(\la)$. Thus we construct new absolutely continuous invariant
ergodic measures. Repeating this procedure, we find absolutely
continuous invariant ergodic measures such that almost every point in
$Z(\la)$ is in the basin of one of these measures. The number of
measures is finite since the basins are invariant sets and Lemma
\ref{invariantsets} holds. It finishes the proof of Theorem
\ref{PrincipalA}.
\end{proof}
\end{subsection}
\end{section}

\section*{}

\end{document}